\documentclass[12pt]{article}

\usepackage{graphicx}
\usepackage{fullpage}
\usepackage{amsthm}
\usepackage{amssymb}
\usepackage{amsmath}
\usepackage{url}
\usepackage{caption}
\usepackage{subcaption}
\usepackage{pgfplots}
\pgfplotsset{compat=1.16}
\usepackage{eucal}
\usetikzlibrary{calc}

\newcommand{\eps}{\varepsilon}
\newcommand{\abs}[1]{\left|#1\right|}
\newcommand{\ur}[1]{\mathrm{#1}}
\renewcommand{\d}{\ur{d}}
\renewcommand{\a}{\alpha}
\newcommand{\bb}[1]{\mathbb{#1}}
\newcommand{\N}{\bb N}
\newcommand{\Id}{{\sf Id}}
\newcommand{\il}[1]{{#1}_i^\ell}

\newcommand{\p}{\partial}
\newcommand{\pO}{\p \O}

\newcommand{\hil}[1]{\hat{#1}_i^\ell}
\newcommand{\ik}[1]{{#1}_i^k}
\newcommand{\iko}[1]{{#1}_i^{k+1}}
\newcommand{\ia}[1]{{#1}_i^{\alpha}}
\renewcommand{\O}{{\Omega}}
\newcommand{\iao}[1]{{#1}_{\O, i}^{\alpha}}
\newcommand{\bia}[1]{\bar{#1}_i^{\alpha}}
\newcommand{\biao}[1]{\bar{#1}_{\O, i}^{\alpha}}
\newcommand{\bbia}[1]{\bar{\bar{#1}}_i^{\alpha}}
\newcommand{\bbiao}[1]{\bar{\bar{#1}}_{\O, i}^{\alpha}}
\newcommand{\bian}[1]{\bar{#1}_{\pO,i}^{\a}}

\newcommand{\bbian}[1]{\bar{\bar{#1}}_{\pO,i}^{\a}}

\newcommand{\oO}{\overline{\O}}
\newcommand{\R}{\bb R}
\newcommand{\tia}[1]{\tilde{#1}_{\O,i}^{\a}}
\newcommand{\ttia}[1]{\tilde{\tilde{#1}}_{\O,i}^{\a}}
\newcommand{\si}[1]{{\mathsf{#1}}_i}
\newcommand{\sih}[1]{\hat{\mathsf{#1}}_i}
\newcommand{\sia}[1]{{\mathsf{#1}}_i^{\alpha}}
\newcommand{\siha}[1]{{\hat{\mathsf{#1}}}_i^{\alpha}}

\newcommand{\siao}[1]{{\mathsf{#1}}_{\O, i}^{\alpha}}
\newcommand{\siar}[1]{{\mathsf{#1}}^{\alpha}_{\pO, i}}
\newcommand{\siwk}[1]{{\mathsf{#1}}_i^{k,w}}
\newcommand{\sihwk}[1]{{\hat{\mathsf{#1}}}_i^{k,w}}

\newcommand{\bialo}[1]{\bar{#1}_{\O, i}^{\alpha,\ell}}

\newcommand{\bbialo}[1]{\bar{\bar{#1}}_{\O, i}^{\alpha,\ell}}

\newcommand{\smnorm}[1]{\lVert#1\rVert}
\newcommand{\norm}[1]{\left\|#1\right\|}

\newcommand{\ika}[1]{{#1}_i^\kappa}
\newcommand{\ila}[1]{{#1}_i^\lambda}
\newcommand{\tends}{\rightarrow}
\newcommand{\Fg}{\mathcal{F}}
\newcommand{\bx}{\mathbf{x}}
\newcommand{\ba}{{\boldsymbol\alpha}}

\DeclareMathOperator{\supp}{supp}

\renewcommand{\P}{{\sf \P}}
\newcommand{\pair}[2]{\langle #1,#2 \rangle}

\theoremstyle{plain}
\newtheorem{assumption}{Assumption}
\newtheorem{pseudoalgorithm}{Algorithm}
\newtheorem{definition}{Definition}
\newtheorem{remark}{Remark}
\newtheorem{lemma}{Lemma}
\newtheorem{theorem}{Theorem}

\begin{document}
\title{\bf Finite Element Approximation of Hamilton-Jacobi-Bellman equations with nonlinear mixed boundary conditions}


\author{%
{\sc
Bartosz Jaroszkowski\thanks{Email: B.Jaroszkowski@sussex.ac.uk}
and
Max Jensen\thanks{Email: M.Jensen@sussex.ac.uk}} \\[2pt]
Department of Mathematics, University of Sussex, Brighton, UK
}

\maketitle

\begin{abstract}
{We show strong uniform convergence of monotone P1 finite element methods to the viscosity solution of isotropic parabolic Hamilton-Jacobi-Bellman equations with mixed boundary conditions on unstructured meshes and for possibly degenerate diffusions. Boundary operators can generally be discontinuous across face-boundaries and type changes. Robin-type boundary conditions are discretised via a lower Dini derivative. In time the Bellman equation is approximated through IMEX schemes. Existence and uniqueness of numerical solutions follows through Howard's algorithm.}
{Finite element method; Hamilton-Jacobi-Bellman equation; Mixed boundary conditions; Fully nonlinear equation; Viscosity solution}
\end{abstract}

\section{Introduction}
\label{sec:intro}
The value function of an optimal control problem is, under suitable assumptions, the solution of a Hamilton-Jacobi-Bellman (HJB) equation. In this work we consider the numerical solution of HJB equations with mixed boundary conditions of the form:
\begin{subequations}\label{eq:BellmanibvpIntro}
\begin{alignat}{2}
-\p_t v + \sup_{\a\in A} (L^{\a} \;\; v - f^{\a})    & = 0 	&	&\qquad\text{in }[0,T)\times\O,\label{eq:BellmanpdeIntro}\\
-\p_t v + \sup_{\a\in A} \bigl( L^{\a}_{\pO} v - g^{\a} \bigr)    & = 0 	&	& \qquad\text{on }[0,T)\times\pO_t,\label{eq:BellmanrobinboundIntro}\\
\sup_{\a\in A}\bigl( L^{\a}_{\pO} v - g^{\a} \bigr)               & = 0	&	& \qquad\text{on }[0,T)\times\pO_{R},\label{eq:BellmanneumannboundIntro}\\
v - g \;\;\;  & = 0	&	&\qquad\text{on }[0,T)\times\pO_{D},\label{eq:BellmanDiriboundIntro}\\[2mm]
v - v_T \,    & = 0	&	&\qquad\text{on }\{T\}\times\oO.\label{eq:BellmanfinalIntro}
\end{alignat}
\end{subequations}
Here $L^{\a}$ and $L^{\a}_{\pO}$ denote operators on the domain $\Omega$ and its boundary, respectively. The sets $\pO_t$, $\pO_R$ and $\pO_D$ form a decomposition of $\partial \Omega$. While leaving further details of the notation to the next section, it is already apparent how the basic fully nonlinear structure of the PDE operator, meaning the left-hand side of \eqref{eq:BellmanpdeIntro}, is mirrored in the Robin-type boundary conditions \eqref{eq:BellmanrobinboundIntro} and~\eqref{eq:BellmanneumannboundIntro}. But there is a crucial, additional complication of the boundary operators $L^{\a}_{\pO}$: They will in general depend on the full gradient $\nabla v$ and not just on the tangential gradient $\nabla_{\!\pO} v$, meaning that $L^{\a}_{\pO} v$ cannot be evaluated with knowledge of $v|_{\pO}$ only.

Recalling the connection between optimal control and HJB equations, Bellman-type equations as in  \eqref{eq:BellmanrobinboundIntro} naturally arise on sections $\pO_t$ of the boundary. Indeed, the boundary condition \eqref{eq:BellmanrobinboundIntro} expresses the possibility of controlling the particle or agent on the boundary through processes which are implicitly described by the $L^{\a}_{\pO}$. In contrast, Dirichlet conditions \eqref{eq:BellmanDiriboundIntro} are appropriate for those sections $\pO_D$ of $\pO$ where the possibility to control may cease in exchange for the reward or cost of $g_D$.

Boundary conditions of type \eqref{eq:BellmanneumannboundIntro} arise from the Skorokhod control problem, which models particle reflection at the boundary \cite{Lions85, FDbook, Serea03}. Moreover, they have recently been used for the numerical solution of optimal transport problems in the setting of Monge-Amp\`{e}re equations, where the transport boundary conditions are examined in Hamilton-Jacobi form \cite{Benamou12, oblique2}.

For the authors the problem of primary interest is the Heston model of financial interest rates with an uncertain market price of volatility risk \cite{HestonPaper}. The Heston equation is most naturally posed on an unbounded domain, where already with a certain market price of volatility risk it appears with mixed boundary terms corresponding to \eqref{eq:BellmanrobinboundIntro}, \eqref{eq:BellmanneumannboundIntro} as well as \eqref{eq:BellmanDiriboundIntro}. All those types of boundary conditions remain when introducing uncertainty and when truncating the domain for the purposes of numerical approximation.

The aim of this work is to introduce a finite element method capable of computing approximations to viscosity solutions for the aforementioned problems. This paper extends results of \cite{max_SIAM} with the inclusion of mixed, fully nonlinear boundary conditions. The presented method permits degenerate diffusions. Boundary operators may exhibit discontinuities across face boundaries and where the type of boundary condition changes. 

A challenge for problems of this type is the discretisation of the first-order directional derivatives in \eqref{eq:BellmanrobinboundIntro} and~\eqref{eq:BellmanneumannboundIntro} which needs to be simultaneously consistent and monotone. On the one hand establishing monotonicity with an artificial diffusion approximating the Laplace-Beltrami operator of $\pO$ would not be sufficient because of the normal component in the directional derivatives of \eqref{eq:BellmanrobinboundIntro} and~\eqref{eq:BellmanneumannboundIntro}. On the other hand an artificial diffusion approximating the Laplace operator of $\Omega$ would not vanish under refinement due to different scaling of boundary and domain terms, thus leading to an inconsistent method. Our formulation is based on the observation that lower Dini directional derivatives exist for all functions in the P1 approximation space whenever the direction in question is in the tangent cone of $\Omega$ at the position of interest.

A benefit of the finite element approach is that besides $L^\infty$ also $L^2(H^1)$ convergence can be established on unstructured meshes, as was shown in \cite{max_SIAM,IJEN} for the Dirichlet problem. The $L^2(H^1)$ convergence is for instance important for the above mentioned Heston model \cite{HestonPaper} as Delta hedging requires knowledge of partial derivatives of the value function. 

The numerical analysis of HJB equations with Neumann and Robin conditions encompasses only few works. First results were provided by the finite difference community; we refer to the text book \cite{FDbook}. More recently, the transport boundary conditions of optimal transport were in \cite{Benamou12} approximated with a filtered wide stencil scheme.  In \cite{SLnonlinear} a nonlinear Neumann boundary operator is approximated by extending the boundary into a strip of positive thickness, allowing the boundary conditions to be treated like a PDE operator. Within the finite element setting one line of research has developed around the approximation of Cordes solutions, in \cite{oblique1} with a mixed, non-conforming finite element method while in \cite{oblique2} with a discontinuous Galerkin finite element method. Both \cite{oblique1} and \cite{oblique2} concentrate on the linear setting in non-divergence form. For a general review of the approximation of fully nonlinear equations with other types of boundary conditions we refer to \cite{review1, review2}.

The structure of this article is as follows. In Section \ref{sec:bvp} we formulate the HJB problem with mixed boundary conditions. In Section \ref{sec:num} we define the numerical method. In Section \ref{sec:mono} we prove monotonicity properties of the discretised operators. In Section \ref{sec:exis} we show the existence and uniqueness of numerical solutions. In Sections \ref{sec:cons} and \ref{sec:stab} we establish consistency and stability, respectively, leading us to our main result of convergence in Section \ref{sec:conv}. Finally, we present numerical experiments in Section \ref{sec:numexp}.

\section{Mixed final time boundary value Bellman problem} \label{sec:bvp}

In this section we introduce time-dependent Bellman equations with mixed boundary conditions. We consider a polytopic domain $\O \subset \R^d$ with $d \geq 2$, i.e.~a bounded, connected, closed domain, whose interior is non-empty and whose boundary is formed of flat faces. We allow $\Omega$ to be non-convex. Let $\mathcal{F}_k$ denote set of open $k$-dimensional faces of $\Omega$ contained in $\pO$. 

We consider Dirichlet and Robin boundary conditions on disjoint subsets of boundary $\pO$. Additionally, the region of the Robin boundary conditions breaks into two parts, one with and one without time derivative. We denote those three disjoint regions as $\pO_D$, $\pO_R$ and $\pO_t$, respectively. Therefore $\pO_D \cap \pO_R \cap \pO_t = \emptyset$ and $\pO_D \cup \pO_R \cup \pO_t = \pO$. 

It is convenient to define the notion of a {\em generalised face} as the intersection of an $\omega' \in \mathcal{F}_{d-1}$ and a region linked to a boundary condition:
\[
\Fg := \bigl\{ \omega \subset \pO : \, \omega = \omega' \cap \pO_X \text{ where } \omega' \in \mathcal{F}_{d-1}, \pO_X \in \{ \pO_D, \pO_R, \pO_t \} \bigr\}
\]
We assume that the boundary conditions are continuous on each $\omega$; however, discontinuities across (generalised) face boundaries may occur.

We introduce the normed space of piecewise continuous functions
\[
PC(\pO, \R^k) := \{ g \in L^{\infty}(\pO, \R^k) : \; g |_{\omega} \in C(\omega, \R^k) \quad \forall \, \omega \in \Fg \},
\]
equipped with the $L^\infty(\pO, \R^k)$ norm. If $k = 1$, we simply write $PC(\pO)$. We denote the standard inner product of $L^2 (\O)$ and $L^2 (\O, \R^d)$ by $\langle \cdot, \cdot \rangle$. 

Let $A$ be a compact metric space, $\a \in A$ and let $L^{\a}$ be a linear operators of the following form:
\begin{equation*}\label{linop}
L^{\a} : \; \oO \times \R \times \R^d \times \R \to \R, \; (x,q,p,s) \mapsto - a^{\a}(x) \, q - b^{\a}(x) \cdot p + c^{\a}(x) \, s.
\end{equation*}
The interpretation as differential operator follows with $q = \Delta w(x)$, $p = \nabla w(x)$ and $s = w(x)$ for $w \in C^2(\oO)$. The mapping
\[
 A \to C(\oO) \times C(\oO, \R^d) \times C(\oO) \times C(\oO),\\ \a \mapsto (a^{\alpha},  b^ {\a}, c^ {\a}, f^ {\a}),
\]
is assumed to be continuous such that the families of functions $\{ a^\a \}_{\a \in A}$, $\{ b^\a \}_{\a \in A}$, $\{ c^\a \}_{\a \in A}$ and $\{ f^\a \}_{\a \in A}$ are equicontinuous. We require that $a^{\a}(x) \geq 0$ for all $\a \in A$ so that all $L^{\a}$ are degenerate elliptic. Frequently we abbreviate $L^\a(x,\Delta w(x),\nabla w(x),w(x))$ by $L^\a w(x)$ and $x \mapsto L^\a w(x)$ by $L^\a w$.

For $\a \in A$ the Robin operators $L^\a_{\pO}$ are defined as
\begin{align} \label{eq:LBC}
L^{\a}_{\pO} : \pO \times \R^d \times \R \to \R, \; (x,p,s) \mapsto - b_{\pO}^{\a}(x) \cdot p + c_{\pO}^{\a}(x) s
\end{align}
with $b_{\pO}^{\a} \in PC(\pO, \R^d)$, $c_{\pO}^{\a} \in PC(\pO)$. The abbreviations $L^\a_{\pO} w(x)$ and $L^\a_{\pO} w$ are used analogously to $L^\a$.

We can now pose Hamilton-Jacobi-Bellman (HJB) problem, whose numerical solution is the subject of this paper:
\begin{subequations}\label{eq:Bellmanibvp}
\begin{alignat}{2}
-\p_t v + \sup_{\a\in A} (L^{\a} \;\; v - f^{\a})    & = 0 	&	&\qquad\text{in }[0,T)\times\O,\label{eq:Bellmanpde}\\
-\p_t v + \sup_{\a\in A} \bigl( L^{\a}_{\pO} v - g^{\a} \bigr)    & = 0 	&	& \qquad\text{on }[0,T)\times\pO_t,\label{eq:Bellmanrobinbound}\\
\sup_{\a\in A}\bigl( L^{\a}_{\pO} v - g^{\a} \bigr)               & = 0	&	& \qquad\text{on }[0,T)\times\pO_{R},\label{eq:Bellmanneumannbound}\\
v - g \;\;\;  & = 0	&	&\qquad\text{on }[0,T)\times\pO_{D},\label{eq:Bellmandirichletbound}\\[2mm]
v - v_T \,    & = 0	&	&\qquad\text{on }\{T\}\times\oO,\label{eq:Bellmanfinal}
\end{alignat}
\end{subequations}
with $g \in C(\pO)$, $g^{\a} \in PC(\pO)$, $v_T \in C(\oO)$ and $T \in (0,\infty)$. The suprema are applied pointwise. An interpretation of \eqref{eq:Bellmanibvp} in the context of optimal control is given in Appendix \ref{sec:interpretation}. Observe that the data terms $g^{\a}$ of the Robin conditions are $\a$-dependent, while the corresponding Dirichlet data $g$ are not. We assume equicontinuity of mapping
\[
 A \to PC(\pO, \R^d) \times PC(\pO) \times PC(\pO), \; \a \mapsto (b_{\pO}^{\a},  c_{\pO}^{\a}, g^{\a})
\]
in $\a \in A$. Additionally, we require the sign-conditions $v_T, \, c_{\O}^{\a}, \, c_{\pO}^{\a}, \, g,\,g^{\a} \geq 0$. It follows from continuity that
\begin{align} \label{rbounds}
\sup_{\alpha \in A} \| \, (b_{\pO}^{\a}, c_{\pO}^{\a}, g^{\a}) \,
\|_{L^\infty(\pO, \R^d) \times L^\infty(\pO) \times L^\infty(\pO)} < \infty.
\end{align}
We require $v_T$ to satisfy the Dirichlet boundary conditions on~$\pO_D$.

It is useful to formulate the operator used in \eqref{eq:Bellmanibvp} more succinctly as
\begin{align*} \label{def:F}
F(t,x,q,p,r,s) = & \left\{ 
\begin{array}{rllll}
- r + \sup_{\a} \bigl(L^{\a}(x,q,p,s) - f^{\a}(x)\bigr) &\text{if } (t,x) \in [0,T)\times\oO,\\
- r + \sup_{\a} \bigl( L^{\a}_{\pO}(x,p,s) - g^{\a}(x) \bigr) &\text{if } (t,x) \in [0,T)\times\pO_t, \\
\phantom{- r + \;} \sup_{\a}\bigl( L^{\a}_{\pO}(x,p,s) - g^{\a}(x) \bigr) &\text{if } (t,x) \in [0,T)\times\pO_{R}, \\
s-g(x) \;\;\; &\text{if } (t,x) \in [0,T)\times\pO_{D}, \\
s-v_T(x) \; &\text{if } (t,x) \in \{T\}\times\oO.
\end{array}
\right.
\end{align*}

We conclude the section with a definition of a viscosity solution similar to the setting of \cite{barles_souganidis} which will be used throughout the chapter. To this end, let us consider a bounded function $v:[0,T] \times \oO \to \R$ and its upper and lower semi-continuous envelopes, defined respectively as 
\[
v^*(t,x) := \limsup_{\substack{(s,y) \to (t,x) \\ (s,y) \in [0,T] \times \oO}} v(s,y)
\]
and 
\[
v_*(t,x) := \liminf_{\substack{(s,y) \to (t,x) \\ (s,y) \in [0,T] \times \oO}} v(s,y).
\]
We analogously extend the definition of lower- and upper semicontinuous envelopes to $F$.

\begin{definition}\label{def:vissol}
A bounded function $v$ is a viscosity supersolution (respectively, subsolution) of \eqref{eq:Bellmanibvp} if, for any test function $\psi \in C^2(\R \times \R^d)$,
\[
F^*(t, x, \Delta \psi(t,x), \nabla \psi(t,x), \partial_t \psi(t,x), v_*(t,x)) \geq 0,
\]
(respectively, 
\[
\left. F_*(t, x, \Delta \psi(t,x), \nabla \psi(t,x), \partial_t \psi(t,x), v^*(t,x)) \leq 0, \right)
\]
provided that $v_{*}-\psi$ attains a local minimum (respectively, $v^{*}-\psi$ attains a local maximum) at $(t,x) \in [0,T] \times \oO$. Finally, we call $v:[0,T] \times \oO \to \R$ a viscosity solution of \eqref{eq:Bellmanibvp} if it is simultaneously a viscosity sub- and supersolution of \eqref{eq:Bellmanibvp}.
\end{definition}

Here the value of
\[
F^*(t, x, \Delta \psi(t,x), \nabla \psi(t,x), \partial_t \psi(t,x), v_*(t,x))
\]
should only depend on $\psi$'s restriction to $[0,T] \times \oO$. However, generally for lower-dimensional faces $F \in \mathcal{F}_{d-2}$ one finds test functions $\psi, \phi \in C^2(\R \times \R^d)$ with $\psi|_{[0,T] \times \oO} = \phi|_{[0,T] \times \oO}$ such that $\nabla \psi(t,x) \neq \nabla \phi(t,x)$ for $x \in F$. 
 
We therefore demand that the coefficient $b_{\pO}^{\a}(x)$ of \eqref{eq:LBC} belongs to the tangent cone:
\begin{align} \label{eq:conecond}
b_{\pO}^{\a}(x) \in K(x) \qquad \forall \, \a \in A, x \in \pO.
\end{align}
Here, because of the polytopic nature of the domain, we define the tangent cone $K(x)$ as
\begin{align*}
K(x) := \bigl\{ x' \in \R^d \big| \; \exists \, \Lambda \in (0,\infty) \; \forall \, \lambda \in [0,\Lambda] : x + \lambda \, x' \in \oO \bigr\},
\end{align*}
i.e.~$x'$ is in the cone if there is a line segment from $x$ in the direction of $x'$ which is contained in $\oO$. Indeed, for $b_{\pO}^{\a}(x) \in K(x) \setminus \{ 0 \}$ we observe how
\begin{align} \label{eq:dirder}
- b_{\pO}^{\a}(x) \cdot \nabla \psi(t,x) = \partial_{-b_{\pO}^{\a}(x)} \psi(t,x) = \lim_{\substack{\lambda \to 0\\[.3mm] \lambda > 0}} \frac{\psi(t,x) - \psi(t,x + \lambda \, b_{\pO}^{\a}(x))}{\lambda}
\end{align}
is expressed only referring to $\psi$ on $[0,T] \times \oO$ and thus independently of $\psi$'s extension to $\R \times \R^d$. The limit on the right-hand side of \eqref{eq:dirder} is known as the lower Dini derivative of $\psi$ in direction $- b_{\pO}^{\a}(x)$. On smooth sections of the boundary and at outward pointing corners the requirement \eqref{eq:conecond} corresponds to an outflow condition, while at re-entrant corners \eqref{eq:conecond} may permit an inflow term. In that sense, \eqref{eq:conecond} is less restrictive than oblique boundary conditions such as \cite[(7.35)]{Crandall:1992ta}. We remark, however, that strengthened versions such as \cite[(7.35)]{Crandall:1992ta} may be necessary to ensure the existence of a comparison principle for the final time boundary value problem of the specific application of interest.

\section{Numerical scheme} \label{sec:num}

For the discretisation of \eqref{eq:Bellmanibvp} we consider a sequence $V_i, i \in \N$, of piecewise linear, simplicial, shape-regular finite element spaces. Let $\mathcal{T}_i$ be the mesh corresponding to the finite element space $V_i$. The boundary mesh $\mathcal{B}_i$ consists of the $(d-1)$-dimensional faces $F$ of elements $K \in \mathcal{T}_i$ with $F \subset \pO$. We make the assumption that $\mathcal{B}_i$ is subordinate to $\Fg$, i.e.~every open $F \in \mathcal{B}_i$ is contained entirely in exactly one generalised face~$\omega \in \Fg$.

Let $V_i^g \subset V_i$ be the affine subspace of functions which interpolate the Dirichlet boundary data on $\pO_D$ and $V_i^{0} \subset V_i$ be the vector subspace of functions which interpolate $0$ on $\pO_D$. The nodes of the finite element mesh are denoted by $\il y$. Here the index $\ell$ ranges over the nodes in the interior first, then the nodes on $\pO_t$, then $\pO_R$ and finally $\pO_D$. Therefore, $\il y \in \O$  for $\ell \leq N_i^{\O}$ for some $N_i^{\O} \in \N$ denoting number of interior nodes, $\il y \in \O \cup \pO_t$ for $\ell \leq N_i^{t}$ for some $N_i^{t} \in \N$ and, lastly, $\il y \in \O \cup \pO_t \cup \pO_{R}$ for $\ell \leq N_i := \dim V_i^g$. These nodes $\il y \in \O \cup \pO_t \cup \pO_{R}$ are called non-Dirichlet nodes.

The associated hat functions $\il \phi \in V_i$ are chosen so that $\il \phi(\il y) = 1$ while $\il \phi(y_i^s) = 0$ for $\ell \neq s$. Set $\hil \phi := \il \phi / \| \il \phi \|_{L^1(\O)}$. Therefore, the $\il \phi$ are normalised in the $L^\infty(\oO)$ norm whilst the $\hil \phi$ are normalised in the $L^1(\oO)$ norm.

The mesh size, i.e. the largest diameter of an element, is denoted $\Delta x_i$. It is assumed that $\Delta x_i \to 0$ as $i \to \infty$. The uniform time step size is denoted $h_i$ with the constraint that $T / h_i \in \N$.  It is assumed that $h_i \to 0$ as $i \to \infty$. Let $\ik s$ be the $k$th time step at the refinement level~$i$. Then the set of time steps is $S_i := \bigl\{ \ik s : k = T/h_i, \dots, 0 \bigr\}$. 

We introduce the operator $d_i$, which approximates the time derivative on $\Omega$ and $\pO_t$ but which is $0$ on the remaining boundary. More precisely, we let the $\ell$th entry of $d_i w( \ik s, \cdot)$ be
\begin{equation*}
(d_i w( \ik s, \cdot))_\ell = 
\begin{cases} \frac{w( \iko s, \il y) - w(\ik s, \il y)}{h_i} \quad &\ell \leq N_i^{t},
\\
\hspace{0.12\textwidth} 0 &\textrm{otherwise.}
\end{cases}
\end{equation*}
Observe how $(d_i w( \ik s, \cdot))_\ell = 0$ for nodes $\il y \in \pO_R$ is consistent with the structure of~\eqref{eq:Bellmanneumannbound}.

For each discretisation of \eqref{eq:Bellmanibvp} we allow a splitting of $L^{\a}$ and $L^{\a}_{\pO}$ into an explicit and an implicit part. For each $\alpha$ and for each $i$, we introduce the explicit operator $\iao E$ and the implicit operator $\iao I$ such that
\begin{align*}
\iao E &: \; C^2(\oO) \to C(\oO), \; w \mapsto - \biao  a \, \Delta w \, - \biao  b \cdot \, \nabla w + \biao  c \, w,\\
\iao I &: \; C^2(\oO) \to C(\oO), \; w \mapsto - \bbiao a \, \Delta w \, - \bbiao b \cdot \, \nabla w + \bbiao c \, w,
\end{align*}
where $\biao a, \bbiao a, \biao c, \bbiao c \in C(\oO)$ and $\biao b, \bbiao b \in C(\oO, \R^d)$. Assumption~\ref{ass:consistency} below shows that the explicit and implicit operators are chosen such that $\iao I + \iao E$ approximates $L^{\a}$. Analogously, we introduce non-negative $f_i^{\a} \in C(\oO)$ which approximate $f^{\a}$. 

The discretisations of $\iao E$ and $\iao I$ are the mappings from $V_i$ to $\R^{N_i}$ which are given by
\begin{subequations}\label{eq:discreteopint}
\begin{align}
(\siao E w)_\ell := & \, \biao  a(\il y) \langle \nabla w, \nabla \hil \phi \rangle + \langle -\biao  b \cdot \nabla w + \biao  c \, w, \hil \phi \rangle,\\
(\siao I w)_{\ell} := & \, \bbiao a(\il y) \langle \nabla w, \nabla \hil \phi \rangle + \langle - \bbiao b \cdot \nabla w + \bbiao c w, \hil \phi \rangle, \label{sch:impo}\\
(\siao F)_\ell := & \, \langle \ia f, \hil \phi \rangle,
\end{align}
\end{subequations}
where $\ell$ ranges over all internal nodes,~i.e. $\ell \leq N^\Omega_i$. For boundary nodes $\ell > N^\Omega_i$ we set $(\siao E w)_\ell = (\siao I w)_\ell = (\siao F)_\ell = 0$. Because of the scaling of the $\hil \phi$, integration-by-parts gives for smooth $w$ and large $i$ that $\langle \nabla w, \nabla \hil \phi \rangle \approx - \Delta w(x)$ if $x \approx \il y$ away from the boundary. 

Similarly, we define operators $\siar E$ and $\siar I$ on the boundary to discretise $L^{\a}_{\pO}$ as the sum of an explicit and implicit part. Starting point is the observation that the directional derivative $\partial_{-b_{\pO}^{\a}} w$ is well-defined in the sense of \eqref{eq:dirder} for functions $w \in V_i$ even though $w$ is in general not differentiable.

For interior nodes with index $0 \le \ell \leq N_i^\Omega$ we set $(\siar E w)_\ell = (\siar I w)_\ell = (\siar F)_\ell = 0$. More interestingly, for $N_i^\Omega < \ell \leq N_i$ ranging over the nodes of the Robin boundary conditions, we define the mappings from $V_i$ to $\R^{N_i}$ by
\begin{subequations}\label{eq:discreteopbound}
\begin{align}
(\siar E w)_\ell & := \partial_{-\bian  b(\il y)} w(\il y) + \bian  c(\il y) \, w(\il y),\\
(\siar I w)_\ell & := \partial_{-\bbian  b(\il y)} w(\il y) + \bbian  c(\il y) \, w(\il y),\\
(\siar F)_\ell & := \ia g(\il y),
\end{align}
\end{subequations}
where $\bian c, \bbian c, \ia g \in PC(\pO)$ and $\bian b, \bbian b \in PC(\pO, \R^d)$. Here $\partial_{-\bian  b}$ and $\partial_{-\bbian  b}$ are understood as lower Dini derivatives as in \eqref{eq:dirder}. 

On the Dirichlet boundary the mappings $\siar E$, $\siar I$ and $\siar F$ implement nodal interpolation. For $\ell > N_i$ we set
\begin{subequations}\label{eq:discreteopDir}
\begin{align}
(\siar E w)_\ell & := 0,\\
(\siar I w)_\ell & := w(\il y),\\
(\siar F)_\ell & := g(\il y).
\end{align}
\end{subequations}

We assume a fully implicit discretisation of the region $\pO_R$; additionally, suppose that $\bbian c$ is chosen positive on $\pO_R$, even if $c_{\pO}^{\a}=0$. 

In summary, we require that the following assumption holds.
\begin{assumption} \label{ass:consistency}
The coefficients satisfy
\begin{align*}
\lim_{i \to \infty} \sup_{\a\in A} \Bigl( & \sup_{0 \le \ell \le N_i} \bigl\| a^\alpha - \bigl( \biao a(\il y) + \bbiao a(\il y) \bigr) \bigr\|_{L^\infty({\rm supp} \, \hil \phi)} \\
& + \bigl\| b^\alpha - \bigl( \biao b + \bbiao b \bigr)  \bigr\|_{L^\infty(\O,\R^d)}
 +  \bigl\| c^\alpha - \bigl( \biao c + \bbiao c \bigr) \bigr\|_{L^\infty(\O)} \\
& +\bigl\| f^\alpha - \ia f \bigr\|_{L^\infty(\O)}  \Bigr) = 0
\end{align*}
and
\begin{align*}
\lim_{i \to \infty} \sup_{\a\in A} \Bigl( & \bigl\| b_{\pO}^{\a} - \bigl( \bian b + \bbian b \bigr) \bigr\|_{L^\infty(\pO)} + \bigl\| c_{\pO}^{\a} - \bigl( \bian c + \bbian c \bigr) \bigr\|_{L^\infty(\pO)}\\
& + \bigl\| f^{\a} - f_i^{\a} \bigr\|_{L^\infty(\pO)} + \bigl\| g - g_i \bigr\|_{L^\infty(\pO)} + \bigl\| g^{\a} - g_i^{\a} \bigr\|_{L^\infty(\pO)} \Bigr) = 0.
\end{align*}
We require that the family
\begin{align*} 
\begin{array}{ll}
\{ (\biao a, \biao b, \biao c, \bian b, \bian c, \bbiao a, \bbiao b, \bbiao c, \bbian b, \bbian c, f_i^{\a}, g_i^{\a}) \}_{\a \in A}
\end{array} \end{align*}
is equicontinuous and depends continuously on $\a$. We impose $\biao a = \biao c = \bian c = 0 \in \R$ and $\biao b = \bian b = 0 \in \R^d$ as well as $\bbian c > 0$ on the restriction to $\pO_R$, $i \in \N$.
\end{assumption}

We define 
\[
\sia E = \siao E + \siar E, \quad \sia I = \siao I + \siar I, \quad \sia F = \siao F + \siar F.
\]
We also use the notation $\sia I$,\,$\sia E$ and $\sia F$ for the matrix representations of exactly these $\sia I$,\,$\sia E$ and $\sia F$ with respect to the nodal basis $\{\il{\phi}\}_\ell$ for the trial functions. Moreover, we assume that the supremum operator is applied componentwise, i.e.~$(\sup_{\a}v^{\a})_\ell = \sup_{\a}v_\ell^{\a}$ for $v \in \R^n$. The expression $a \lesssim b$ means that there exists a generic constant $C > 0$, independent of $i$ and $\a$, such that $a \le C b$. Relation $a \gtrsim b$ is defined analogously.

We can now state the numerical scheme used to approximate the solution of~\eqref{eq:Bellmanibvp}. We initialise the scheme by the nodal interpolation of $v_T$ so that $v_i(T, \cdot) \in V_i^g$. Then, in order to find the numerical solution $v_i (\ik s, \cdot) \in V_{i}^g$, we proceed inductively over the remaining timesteps $k \in \left\{T/h_i-1, \dots, 1, 0\right\}$:

\begin{align} \label{eq:numsol}
- d_i v_i(\ik s, \cdot) +  \sup_{\alpha\in A} \bigl( \sia E v_i(\iko s,\cdot) + \sia I v_i(\ik s,\cdot) - \sia F \bigr) = 0.
\end{align}

We also use an alternative formulation of the numerical scheme. The matrices $\siwk E$, $\siwk I$ and $\siwk F$ are constructed row-wise out of the matrices $\sia E$, $\sia I$ and $\sia F$. More precisely, given a node $\il y$, timestep $\ik s$ and a function $w(\ik s, \cdot) \in H^1(\oO)$, let $\hat{\a}$ be a maximiser of 
\[
\sup_{\alpha\in A} \bigl( \sia E w(\iko s,\cdot) + \sia I w(\ik s,\cdot) - \sia F \bigr)_\ell = 0.
\]
Note that choice of $\hat{\a}$ is not necessarily unique; the analysis is valid for any choice of such $\hat{\a}$. We let the $\ell$th row of $\siwk E$, $\siwk I$ and $\siwk F$ be equal to the $\ell$th row of $\mathsf{E}_i^{\hat{\alpha}}$, $\mathsf{I}_i^{\hat{\alpha}}$ and $\mathsf{F}_i^{\hat{\alpha}}$, respectively. In a non-ambiguous case we will omit explicit mention of $k$ and simply write $\si E^w$, $\si I^w$ and $\si F^w$. We can now reformulate \eqref{eq:numsol} using the newly constructed operators. We initialize the scheme with the interpolant $v_i(T, \cdot)$. Then $v_i \in V_i^g$ for each~$k \in \{T/h_i-1, \dots, 1,0\}$ and for $0 \leq \ell \leq N_i^{t}$ solves
\begin{subequations}\label{eq:_alnumsol}
\begin{equation}
\Bigl( (h_i \si I^{k,v_i} + \Id)\, v_i(\ik s, \cdot) + (h_i \si E^{k,v_i} - \Id)\, v_i(\iko s, \cdot) - h_i \si F^{k,v_i} \Bigr)_\ell = 0 \\
\end{equation}
and for each~$k \in \{T/h_i-1, \dots, 1,0\}$ and for $N_i^t < \ell$ solves
\begin{equation}
\Bigl( h_i \si I^{k,v_i} \, v_i(\ik s, \cdot) - h_i \si F^{k,v_i} \Bigr)_\ell = 0,
\end{equation}
\end{subequations}
recalling the implicit discretisation on $\pO_R \cup \pO_D$, enforced through \eqref{eq:discreteopDir} and Assumption~\ref{ass:consistency}.

For the sake of convenience let us also introduce the operators $\sih I^{k,v_i}$, $\sih E^{k,v_i}$ and $\sih F$ which combine spatial and temporal terms. The $\ell$th row of $\sih I^{k,v_i}$, $\sih E^{k,v_i}$ and $\sih F$ is equal to that of $(h_i \si I^{k,v_i} + \Id)$, $(h_i \si E^{k,v_i} - \Id)$ and $h_i \si F^{k,v_i}$, respectively, if $\leq \ell \leq N_i^{t}$. If $N_i^t < \ell$, the $\ell$th row is equal to $(h_i \si I^{k,v_i})$, a zero vector and $h_i \si F^{k,v_i}$, respectively. For a fixed control $\a$, the operators $\siha I$, $\siha E$ and $\siha F$ are constructed in an analogous manner. Then for all timesteps $\ik s$ and each node $\il y$ solution $v_i$ of \eqref{eq:_alnumsol} solves also:
\begin{equation}\label{eq:alnumsol}
    \sih I^{k,v_i}\, v_i(\ik s, \il y) +\sih E^{k,v_i}\, v_i(\iko s,\il y) - \sih F^{k,v_i} = 0.
\end{equation}

\begin{remark}
To implement the lower Dini derivative $\partial_{-\bian  b(\il y)} w(\il y)$ in a computer code, we note that for $\lambda > 0$ sufficiently small there is an element $K \in \mathcal{T}_i$ whose closure contains both $\il y$ and $\il y + \lambda \, \bian b(\il y)$. Indeed, choosing $\lambda$ such that $\lambda \, \bian b(\il y)$ is smaller than the smallest element edge diameter for all $N_i^{\O} < \ell \leq N_i$ achieves this. We then have
\[
\partial_{- \bian b(\il y)} w(\il y) = \frac{w(\il y) - w(\il y + \lambda \, \bian b(\il y))}{\lambda}.
\]
Importantly, because $w \in V_i$ is affine on $\overline{K}$, even without taking a limit $\lambda \to 0$ as on the right-hand side of \eqref{eq:dirder} the Dini derivative is obtained exactly.
\end{remark}

\section{Monotonicity} \label{sec:mono}

In this section we consider monotonicity properties of the discrete differential operators defined in the previous section. Monotonicity is crucial for proving the existence of a unique numerical solution of \eqref{eq:numsol} as well as for establishing convergence to the viscosity solution. 

\begin{definition} \label{def:lmp_wdmp}
Let us consider $v \in V_i$ that has a local non-positive minimum at a node $\il y$. We say that an operator $F$ satisfies a Local Monotonicity Property (LMP) if for any such $v$ it follows that $(Fv)_\ell \leq 0$. Additionally, the operator $F$ satisfies the weak Discrete Maximum Principle (wDMP) provided that, for any $v \in V_i$,
\begin{equation}\label{eq:wDMP}
\bigl( F v \bigr)_{\ell} \geq 0 \;\; \forall \, \ell \in \{1,\dots, N_i\} \;\; \implies \;\; \min_{\O \cup \pO_{R}\cup \pO_t} \!\! v \geq \min \{ \min_{\pO_D} v, 0 \}.
\end{equation}
\end{definition}

We now describe a method for choosing the artificial diffusion coefficients to impose the LMP on the matrices $\sia I$ and $\sia E$. It is based on the assumption of strict acuteness on the mesh. Consider an element $K \in \mathcal{T}_i$ with diameter $\Delta x_K$. For a bounded function $g: \oO \to \R^d$ we define $g$'s norm on the restriction to $K$ as
\[
|g|_K:= \Bigl( \sum_{j=1}^{d} \Bigl\| g_j \Bigr\|_{L^\infty(K)}^2 \Bigr)^{\frac{1}{2}}.
\]
Then by strict acuteness of the meshes we mean that there exists a $\theta \in (0, \frac{\pi}{2})$ such that the following holds:
\begin{equation} \label{eq:strict_acuteness}
\nabla \il \phi \cdot \nabla \phi_{i}^{l} \bigl|_K \leq - \, \sin(\theta) \; | \nabla \il \phi |_K \; | \nabla \phi^{l}_i |_K \qquad \forall \ell,l \leq N_i,\;\ell \neq l, \; \forall K\in\mathcal{T}_i.
\end{equation}
We say that the family of meshes $\{ \mathcal{T}_i \}_i$ is uniformly strictly acute if $\theta$ does not depend on $i$. As discussed in \cite{BUR}, for $d=2$ and $d=3$ the angle $\theta$ can be interpreted geometrically as $\frac{\pi}{2}$ minus the largest angle between the pairs of $(d-1)$-dimensional faces of the element~$K$. 

\subsection{The LMP of $\siao E$, $\siar E$, $\siao E$ and $\siar I$}\label{subsec:LMPbound}

Let the functions $\tia a, \ttia a, \biao c, \bbiao c \in C(\oO)$ and $\biao b, \bbiao b \in C(\oO, \R^d)$ be given. These functions may be chosen freely as long as Assumption \ref{ass:consistency} holds. Conceptually $\tia a + \ttia a \approx a^\a$ is the splitting of the second-order coefficients into explicit and implicit part {\em without} the addition of artificial diffusion. With the addition of artificial diffusion, the coefficients $\bia a$ and $\bbia a$ of Assumption~\ref{ass:consistency} are obtained.

Indeed, as $\biao b, \bbiao b, \biao c, \bbiao c$ are bounded, we can select non-negative artificial diffusion coefficients $ \bialo \nu $ and $ \bbialo \nu $ so that we have for all interior nodes $\il y$ and mesh elements $K$ with $\il y$ as vertex that
\begin{align} \label{eq:artdiff} \begin{array}{c}
| \biao b |_K \, +  \Delta x_K \|  \biao c \|_{L^\infty(K)} \le \bialo \nu \, \sin(\theta) \, | \nabla \hil \phi |_K \, {\rm vol}(K), \\
| \bbiao b |_K \, +  \Delta x_K \| \bbiao c \|_{L^\infty(K)} \le \bbialo \nu \, \sin(\theta) \, | \nabla \hil \phi |_K \, {\rm vol}(K). 
\end{array} \end{align}
Now, choosing $ \biao a, \bbiao a \in C(\oO)$ such that 
\begin{align} \label{eq:artdiffcoef}
\biao a(\il y) \ge \max \bigl\{ \tia a(\il y), \bialo \nu \bigr\}, \qquad \bbiao a(\il y) \ge\max \bigl\{ \ttia a(\il y), \bbialo \nu \bigr\},
\end{align}
we obtain our splitting of $L^{\a}$ into implicit and explicit part. 

\begin{lemma}
Suppose that the mesh $\mathcal{T}_i$ is strictly acute and that \eqref{eq:artdiff} holds. Then $\siao E$ and $\siao I$ satisfy the LMP for all $\a$.
\end{lemma}

\begin{proof}
The argument from \cite[Section 8]{max_SIAM} for the Dirichlet problem carries over unchanged for local minima at interior nodes of $v$ from Definition \ref{def:lmp_wdmp}. At boundary nodes the LMP is trivially satisfied as $\siao E$ and $\siao I$ vanish there.
\end{proof}

\noindent We now turn to the monotonicity of the discrete boundary operators.

\begin{lemma}
The operators $\siar E$ and $\siar I$ satisfy the LMP for all $\a$.
\end{lemma}

\begin{proof}
Let $w \in V_i$ have a local non-positive minimum at a node $\il y \in \pO_t \cup \pO_R$. Then we find for the lower Dini derivative $\partial_{-\bian  b(\il y)} w(\il y) \le 0$. Also $\bian  c(\il y) \, w(\il y) \le 0$ because $c(\il y) \ge 0$. Hence $\siar E$ admits the LMP. The argument for $\siar I$ is analogous. 
\end{proof} 

\subsection{Monotonicity properties of the $\siwk E$, $\sihwk E$, $\siwk I$ and $\sihwk I$}

Having examined the basic building blocks of the numerical scheme in the previous two subsections, we can now analyse the monotonicity properties of the derived operators $\sihwk E$ and $\sihwk I$ as they appear in formulation \eqref{eq:alnumsol} of the scheme. We summarise the assumptions made so far in the selection of the artificial diffusion coefficients.

\begin{assumption} \label{ass:monotonicity}
Suppose that $\mathcal{T}_i$ is strictly acute and that \eqref{eq:artdiff} and \eqref{eq:artdiffcoef} hold.
\end{assumption}

First we examine the explicit terms.

\begin{lemma}\label{lem:Emonotonicity}
Consider a fixed $w : S_i \times \oO \to \R$ such that $w(\ik s, \cdot) \in V_i$ for all $\ik s \in S_i$. Then the operators $v \mapsto \siwk E v$ satisfy the LMP and their matrix has non-positive off-diagonal entries. For $h_i$ small enough, $\sihwk E$ is monotone, i.e.~all entries of the matrix representation are non-positive.
\end{lemma}

\begin{proof}
For any $i$ and $\a$, $\sia E = \siao E + \siar E$ satisfies the LMP because its summands do. Let us consider a $v \in V_i$ that has a local non-positive minimum at a node $\il y$. There is an $\a \in A$ such that $(\siwk E v)_\ell = (\sia E v)_\ell$. We know $ (\sia E v)_\ell \le 0$ and therefore that $v \mapsto \siwk E v$ satisfies the LMP.

For $j \neq \ell$ the hat function $\phi_i^j$ attains a non-positive minimum at $\il y$. Thus, by the LMP, we have that $(\siwk E \phi_i^j)_\ell \leq 0$. Hence all the off-diagonal entries of $\siwk E$ are non-positive. 

Owing to Assumption 1, the discretization on $\pO_R$ is fully implicit. Thus the rows of $\siwk E$ belonging to the discretization on $\pO_R$ contain only zeros. Similarly, the rows linked to $\pO_D$ vanish, see \eqref{eq:discreteopDir}. All other rows include a term arising from the time derivative; their structure is $(h_i \si E^{k,v_i} - \Id)$. Therefore, if $h_i$ is sufficiently small then $\sihwk E$ is monotone. 
\end{proof}

Now we turn to the implicit terms.

\begin{lemma}\label{lem:Imonotonicity}
Consider a fixed $w : S_i \times \oO \to \R$ such that $w(\ik s, \cdot) \in V_i$ for all $\ik s \in S_i$.  Then the operators $v \mapsto \siwk I v$ satisfy the LMP. Moreover, the $ v \mapsto \sihwk I \,v$ satisfy the wDMP and their matrix representation restricted to $V_i$ are strictly diagonally dominant $M$-matrices.
\end{lemma}

\begin{proof}
Analogously to the proof of Lemma \ref{lem:Emonotonicity}, the $v \mapsto \siwk I v$ satisfy the LMP and their off-diagonal entries are non-positive.

Before showing the wDMP we verify strict diagonal dominance. By construction, $v \equiv -1$ attains a non-positive local minimum at each node. Since $\siwk I$ satisfies the LMP property, we have 
\begin{align} \label{eq:boundweakdom}
0 \geq \left(\siwk I v\right)_{\ell}=-\left(\siwk I \right)_{\ell \ell}-\sum_{j \neq \ell} \left(\siwk I \right)_{\ell j}.
\end{align}
As the off-diagonal entries of $\siwk I$ are non-positive, we conclude the weak diagonal dominance of the $\siwk I$:
\[
\left(\siwk I \right)_{\ell \ell} - \sum_{j\neq \ell} \abs{\left(\siwk I \right)_{\ell j}} \geq 0.
\]
The rows of $\sihwk I$ which discretise on $\Omega$ and $\pO_t$ are equal to the respective rows of the strictly diagonally dominant matrix $h_i \siwk I + \Id$.

By Assumption \ref{ass:consistency} we have $\bbian c > 0$ on $\pO_R$. Then
\begin{align} \label{eq:boundstrictdom}
0 > \left(\siwk I v\right)_{\ell} = \left(\sihwk I v\right)_{\ell} \quad \forall \, \il y \in \pO_R.
\end{align}
Using the same argument as above, but noting the strict inequality of \eqref{eq:boundstrictdom} compared to \eqref{eq:boundweakdom}, we conclude the strict diagonal dominance for rows linked to $\pO_R$. On $\pO_D$ the rows resemble an identity matrix, giving also strict diagonal dominance. It follows that the $\sihwk I$ are invertible $M$-matrices because \cite[Chapter 6, Theorem 2.3, $(M_{35})$]{BER} applies as $\sihwk I$ is a $Z$-matrix.

Finally, consider a $v \in V_i$ with
$
\min_{\O \cup \pO_t \cup \pO_t} \!\! v < \min \{ \min_{\pO_D} v, 0 \}.
$
Let $\il y$ be a non-Dirichlet node, where the negative, global minimum of $v$ is attained. Since $\sihwk I$ is a strictly diagonal dominant $M$-matrix it follows that $(\sihwk I v)_\ell < 0$. Hence $\sihwk I$ admits the wDMP.
\end{proof}

\subsection{Scaling of the artificial diffusion coefficients} \label{sec:scaling}
In order to achieve convergence of the numerical scheme we expect the artificial diffusion coefficients $\bialo \nu$, $\bbialo \nu$ to vanish in the limit $i \to \infty$. 

Suppose that \eqref{eq:strict_acuteness} holds uniformly for some $\theta$. In this subsection we suppose $\bialo \nu$ are chosen quasi-optimally with regard to \eqref{eq:artdiff}, meaning
\begin{align} \label{eq:Oartdiffcoeff}
\bialo \nu & \lesssim \sup \Bigl\{\frac{|  \biao b |_K \, +  \Delta x_K \|  \biao c \|_{L^\infty(K)}}{ \sin(\theta) \, | \nabla \hil \phi |_K \, {\rm vol}(K)} \; \Big| \; K\subset \textrm{supp}\; \il \phi \Bigr\}\\
\bbialo \nu & \lesssim \sup \Bigl\{\frac{|  \bbiao b |_K \, +  \Delta x_K \|  \bbiao c \|_{L^\infty(K)}}{ \sin(\theta) \, | \nabla \hil \phi |_K \, {\rm vol}(K)} \; \Big| \; K\subset \textrm{supp}\; \il \phi \Bigr\}.
\end{align}
Generally, in implementations of the algorithm quasi-optimally is more easily achieved than optimality. Because of shape-regularity of the domain one has $| \nabla \hil \phi |_K \, {\rm vol}(K) \gtrsim \frac{1}{\Delta x_K}$. We conclude that quasi-optimal artificial diffusion coefficients satisfy
\begin{align} \label{eq:artdiffscaling}
{\mathsf O}(\bialo \nu) = {\mathsf O}(\bbialo \nu) = \Delta x_K.
\end{align}

We now turn our attention to the time step restrictions imposed through the quasi-optimality \eqref{eq:Oartdiffcoeff}. Recall that in order for the explicit operators to be monotone we require all their entries in matrix representation to be non-positive. This is satisfied trivially for nodes on $\pO_R \cup \pO_D$ where we use a fully implicit scheme. Therefore let us consider non-positivity of the diagonal terms of $h_i \sia E - \Id$ on the complement $\O \cup \pO_t$.  For $\il y \in \O$ this translates into the condition
\begin{align*}
1 \ge \, & h_i \, \bigl( \biao  a(\il y) \langle \nabla \il \phi, \nabla \hil \phi \rangle + \langle - \biao  b \cdot \nabla \il \phi + \biao  c \, \il \phi, \hil \phi \rangle \bigr)
\end{align*}
and for $\il y \in \pO_t$
\[
1 \ge \, h_i \left(\partial_{-\bian  b(\il y)} \il \phi(\il y) + \bian  c(\il y) \right).
\]
Because
\begin{align*}
\langle \nabla \il \phi, \nabla \hil \phi \rangle & = {\mathsf O}\bigl((\Delta x_K)^{-2} \bigr), \\
\langle \nabla \il \phi, \hil \phi \rangle & = {\mathsf O}\bigl((\Delta x_K)^{-1} \bigr),\\
\langle \il \phi, \hil \phi \rangle & = {\mathsf O}\bigl( 1 \bigr), \\
\partial_{-\bian  b(\il y)} \il \phi(\il y) & = {\mathsf O}\bigl((\Delta x_K)^{-1} \bigr),
\end{align*}
we find $h_i = {\mathsf O}((\Delta x_K)^2)$ if $\| \biao  a \|_\infty > 0$. Otherwise if $\| \biao  b \|_\infty > 0$ or $\| \bian  b \|_\infty > 0$ we have $h_i = {\mathsf O}(\Delta x_K)$ and if $\biao  a$, $\bian b$ and $\biao  b$ vanish but not $\biao c$ or $\bian c$, then $h_i = {\mathsf O}(1)$. If also $\biao c = \bian c = 0$ then there is no restriction on~$h_i$, i.e.~fully implicit discretisations are monotone for any $h_i > 0$.

\section{Existence of numerical solutions} \label{sec:exis}

The discrete non-linear problem \eqref{eq:alnumsol} can be solved by a version of Howard's algorithm discussed in \cite{Howard}. We now present its formulation in our setting.

\begin{pseudoalgorithm}
\label{alg}
Given are timestep $k \in \{0,\dots,T/h_i-1\}$, solution $v_i(s_i^{k+1},\cdot) \in V_i$ at timestep $k+1$ and an (arbitrary) choice of $\a \in A$. Find $w_0 \in V_i$ such that
\[
\siha I w_0 = \siha F - \siha E v_i(s_i^{k+1},\,\cdot).
\]
Inductively over $m \in \N$, compute $w_{m+1} \in V_i$ such that

\begin{equation}\label{algeq}
\hat{\mathsf{I}}^{w_m}_i w_{m+1} =  \hat{\mathsf{F}}^{w_m}_i -\hat{\mathsf{E}}^{w_m}_i v_i(s_i^{k+1}, \cdot).
\end{equation}
\end{pseudoalgorithm}

To show the convergence of the sequence $(w_m)_m$ to the solution of \eqref{eq:numsol} we appeal to an auxiliary problem: for some fixed control $\a \in A$ we consider the linear evolution problem associated to it. More precisely, we define $\ia v \colon S_i \to V_i$ to be such that $\ia v(T,\cdot) = v_i(T,\cdot)$, the interpolant of $v_T$, and for each $k \in \{0,\dots,T/h_i-1\}$
\begin{equation}\label{eq:controlnumsol} 
\siha I\, \ia v(\ik s, \cdot) + \siha E\, \ia v(\iko s, \cdot) - \siha F = 0.
\end{equation}
Notice that $\ia v$ is well-defined due to the invertibility of $\siha I$. 

\begin{theorem}\label{thm:discretewellposedness}
There exists a unique numerical solution $v_i \colon S_i \to V_i$ which solves \eqref{eq:numsol} and \eqref{eq:alnumsol}. Algorithm~\ref{alg}, provided with the inputs $k$, $v_i(s_i^{k+1},\cdot)$ and $\a$, generates a sequence $(w_m)_m$ which converges superlinearly to $v_i(\ik s,\cdot)$ as $m\to \infty$. Moreover, $0 \leq v_i \leq \ia v$ for all $\a \in A$.
\end{theorem}

\begin{proof} For a fixed timestep $k$, the superlinear convergence of Algorithm \ref{alg} to the unique solution $v_i(\ik s, \cdot)$ is shown in \cite[Theorem 2.1]{Howard} under Assumptions (H1) and (H2) stated therein. Condition (H1) requires the inverse positivity of the operators $\hat{\mathsf{I}}^{w_m}_i$, which holds because according to Lemma \ref{lem:Imonotonicity} every $\hat{\mathsf{I}}^{w_m}_i$ is a non-singular $M$-matrix. Condition (H2) requires that $\a \in A \mapsto -\siha I$ and $\a \in A \mapsto \siha E\, v_i(\iko s,\cdot) - \siha F$ are continuous, which follows from Assumption \ref{ass:consistency}. Induction over timesteps $k$ gives existence and uniqueness of the solution $v_i$.

We now show that $v_i \geq 0$ on $S_i\times\oO$ by induction over $k$. Firstly, we notice that $v_i(T, \cdot) \geq 0$ because we assumed that $v_T \geq 0$ on $\oO$ and the same holds for its interpolant. Let us assume $v_i(s_i^{k+1},\cdot)\geq 0$ on $\oO$ for some $\iko s \in S_i$. Due to the LMP, all entries of $\sih {E}^{v_i}$ are non-positive and by assumption all entries of $ \sih{F}^{v_i}$ are non-negative. Therefore, using \eqref{eq:alnumsol} we have that
\begin{align*}
\sih{I}^{v_i} v_i(\ik s,\cdot) & = - \sih{E}^{v_i} v_i(\iko s,\cdot) +  \sih{F}^{v_i} \geq 0.
\end{align*}
We conclude that $v_i(s_i^k,\cdot)\geq 0$ on $\oO$ due to the inverse positivity of $ \sih{I}^{v_i}$.

We now prove the last statement, namely $v_i \leq v_i^{\a}$ for all $\a \in A$, by induction over $k$. Consider any $\a \in A$. At time $T$ both $v_i$ and $\ia v$ interpolate $v_T$ and hence are equal. Let us assume that for some $k \leq T/h_i - 1$, $v_i(s_i^{k+1},\cdot) \leq v_i^{\a}(s_i^{k+1},\cdot)$. From \eqref{eq:numsol},
\[
\siha I v_i(s_i^k,\cdot) \leq \siha F -  \siha E v_i(s_i^{k+1},\cdot).
\]
Now subtracting \eqref{eq:controlnumsol} from the above inequality, together with the monotonicity of $\siha E$, gives
\begin{align*}
\siha I \left(v_i(s_i^k,\cdot)-v_i^{\a}(s_i^k,\cdot)\right) &\leq \siha E \left( v_i^{\a}(s_i^{k+1},\cdot)-v_i(s_i^{k+1},\cdot)\right) \leq 0.
\end{align*}
Using the inverse positivity of $ \siha I$ gives us $v_i(s_i^{k},\cdot) - v_i^{\a}(s_i^{k},\cdot) \leq 0$ on $\oO$, as required.
\end{proof}

\section{Consistency} \label{sec:cons}

We will assume existence of an elliptic projection $P_i$, described in \cite{max_SIAM}, with the properties required in the following assumption.

\begin{assumption}\label{ass:ellproj}
There are linear mappings $P_i : C(H^1(\O)) \to V_i$ satisfying for all interior hat functions $\hil \phi$, $\ell \le N_i^\Omega$,
\begin{equation}\label{eq:ellprojdef}
\langle \nabla P_i w, \nabla \hil \phi \rangle = \langle \nabla w, \nabla \hil \phi \rangle.
\end{equation}
There is a constant $C\geq 0$ such that for every $ w \in C^{\infty}(\R^{d})$ and $i\in\N$,
\begin{equation}\label{ass:ellprojstab}
\norm{P_i w}_{W^{1,\infty}(\O)} \leq C \norm{w}_{W^{1,\infty}(\O)}
\quad\text{and}\quad
\lim_{i \to \infty} \norm{P_i w - w}_{W^{1,\infty}(\O)}=0.
\end{equation}
\end{assumption}

To state consistency it is convenient to abbreviate the operator of the numerical scheme as
\begin{align} \label{eq:discreteF}
F_i \, w(\ik s, \il y) := \sih I^{k,w}\, w(\ik s, \il y) +\sih E^{k,w}\, w(\iko s,\il y) - \sih F
\end{align}
for $w(\ik s, \cdot) \in V_i$. Note that while $F_i$ is the discrete operator approximating the continuous operator $F$ defined in section~\ref{sec:bvp}, the notationally similar $\sih F^\alpha$ represents the approximation of $f^\alpha$, $g^\alpha$ and $g$ as explained at the end of section~\ref{sec:num}.
\begin{theorem} \label{thm:consistency}
Let $\psi \in C^2(\R \times \R^d)$, $s_i^{k(i)} \to t \in [0,T)$ and $y_i^{\ell(i)} \to x \in \oO$ as $i\to \infty$. Here $s_i^{k(i)}$ is a time step and $y_i^{\ell(i)}$ a node of the $i$-th refinement. Then
\begin{align} \label{eq:consistency_limsup}
\limsup_{i \to \infty} F_i P_i \psi(s_i^{k(i)}, y_i^{\ell(i)}) \le F^*(t,x, \Delta \psi(t,x), \nabla \psi(t,x), \partial_t \psi(t,x), \psi(t,x))
\end{align}
and
\begin{align} \label{eq:consistency_liminf}
\liminf_{i \to \infty} F_i P_i \psi(s_i^{k(i)}, y_i^{\ell(i)}) \ge F_*(t,x, \Delta \psi(t,x), \nabla \psi(t,x), \partial_t \psi(t,x), \psi(t,x)).
\end{align}
\end{theorem}

\begin{proof} We prove \eqref{eq:consistency_limsup}. The result for \eqref{eq:consistency_liminf} follows analogously. For ease of notation, the dependence of $k$ and $\ell$ on $i$ is made implicit.

{\em Step 1:} Standard finite difference bounds ensure that if $\il y \in \O \cup \pO_t$ then
\begin{equation}\label{eq:timederivconv}
\lim_{i\to \infty} d_i P_i \psi(\ik s, \il y) = \p_t \psi(t,x).
\end{equation}
Otherwise, if $\il y \in \pO_R \cup \pO_D$ then
\begin{equation}\label{eq:timederivconv0}
\lim_{i\to \infty} d_i P_i \psi(\ik s, \il y) = 0.
\end{equation}

{\em Step 2:} It is shown in \cite[Section 4]{max_SIAM} that if $\il y \in \O$ then
\begin{equation}\label{eq:interiorconsistency}
\lim_{i\to \infty} \left( \sia E P_i \psi(\il s,\cdot) + \sia I P_i \psi(\ik s,\cdot)  - \sia F\right)_{\ell} = L^{\a} \psi(t,x) - f^{\a}(x),
\end{equation}
where convergence to the limit is uniform over all $\a \in A$. We remark that the orthogonality \eqref{eq:ellprojdef} is used in this step.

{\em Step 3:} Now suppose that $\il y \in \pO_D$. Then it follows from \eqref{ass:ellprojstab} that
\begin{equation}\label{eq:Dirichletconsistency}
\lim_{i\to \infty} F_i P_i \psi(\ik s,\il y) = \psi(t,x) - g(x).
\end{equation}

{\em Step 4:} Let $\il y \in \pO_t \cup \pO_R$. Just like the continuous operators the corresponding  first-order terms of the discrete Robin operators employ the lower Dini derivative, giving consistency directly. Thus with
\begin{align} \label{eq:ellconsistc}
\lim_{i\to \infty} & \left|c_{\pO}^{\a} \, P_i \psi(t,\cdot) - \bbian c \, P_i \psi(s_i^{k},\cdot) -\bian c \, P_i \psi(s_i^{k+1},\cdot) \right| = 0,
\end{align}
using Assumptions \ref{ass:consistency} and \ref{ass:ellproj}, we conclude that if $\il y \in \pO_t \cup \pO_R$ then
\begin{equation}\label{eq:boundaryconsistency}
\lim_{i\to \infty} \left( \sia E P_i \psi(\il s,\cdot) + \sia I P_i \psi(\ik s,\cdot)  - \sia F\right)_{\ell} = L_{\pO}^{\a} \psi(t,x) - g^{\a}(x).
\end{equation}

{\em Step 5:} Consider the sequence $\{ (\ik s, \il y) \}_i$ as specified in the statement of the theorem, in particular with $\il y \in \oO$. We decompose $\{ (\ik s, \il y) \}_i$ into the subsequences of the $(\ik s, \il y)$ where $\il y$ belongs to $\O$, $\pO_t$, $\pO_R$ and $\pO_D$, respectively. Then the conclusions of Steps 1 to 4 above may be applied to the individual subsequences.
\end{proof}

\section{Stability} \label{sec:stab}

In this section we present a lemma which ensures $L^\infty$ stability of the numerical scheme \eqref{eq:alnumsol}. The stability statement goes back to the boundedness of a supersolution of the continuous linear problem for a fixed $\alpha$. Let
\begin{align*} \label{def:Fa}
F^\alpha (t, x,q,p,r,s) = & \left\{ 
\begin{array}{rllll}
- r + L^{\a}(x,q,p,s) - f^{\a}(x) &\text{if } (t,x) \in [0,T)\times\oO,\\
- r + L^{\a}_{\pO}(x,p,s) - g^{\a}(x) &\text{if } (t,x) \in [0,T)\times\pO_t, \\
\phantom{- r + \;} L^{\a}_{\pO}(x,p,s) - g^{\a}(x) &\text{if } (t,x) \in [0,T)\times\pO_{R}, \\
s-g(x) \;\;\; &\text{if } (t,x) \in [0,T)\times\pO_{D}, \\
s-v_T(x) \; &\text{if } (t,x) \in \{T\}\times\oO.
\end{array}
\right.
\end{align*}

\begin{assumption} \label{ass:stab_existence}
There exists an $\a \in A$ and a $w(t,x) \in C^2(\R \times \R^d)$ which is a strict supersolution of the associated linear problem. More precisely, there is an $\eps > 0$ such that
\[
F^\alpha(t, x,\Delta w(t,x), \nabla w(t,x), \partial_t w(t,x), w(t,x)) \ge \eps
\]
on $[0,T] \times \oO$.
\end{assumption}

The assumption is essentially fulfilled if the linear equation
\begin{align} \label{eq:lin_prob_eps}
F^\alpha(t, x,\Delta \psi(t,x), \nabla \psi(t,x), \partial_t \psi(t,x), \psi(t,x))= 2 \eps
\end{align}
is a well-posed problem in a suitable sense. For example, $\psi$ may be a weak solution of \eqref{eq:lin_prob_eps} as in \cite[Chapter 1]{Oleinik} which admits a bounded extension to $\R \times \R^d$. In such cases one may pass to a strict supersolution in $C^2(\R \times \R^d)$ through mollification.

\begin{lemma}\label{lem:cons_err_convergence}
Let $\a$ be as in Assumption \ref{ass:stab_existence}, $s_i^{k(i)} \to t \in [0,T)$ and $y_i^{\ell(i)} \to x \in \oO$ as $i\to \infty$. Here $s_i^{k(i)}$ is a time step and $y_i^{\ell(i)}$ a node of the $i$-th refinement. Then
\begin{align*} 
\liminf_{i \to \infty} \left[ \siha I P_i w(\ik s, \il y) +\siha E P_i w(\iko s, \il y) - (\siha F)_\ell \right] \ge \eps.
\end{align*}
\end{lemma}

\begin{proof}
We use Theorem \ref{thm:consistency} for a singleton control set $A = \{ \alpha \}$ to cover the linear case. The result now follows because
\begin{align*}
(F^\alpha)_*(x, \Delta w(t,x), \nabla w(t,x), \partial_t w(t,x), w(t,x)) \ge \eps,
\end{align*}
owing to Assumption \ref{ass:stab_existence}. 
\end{proof}

\begin{theorem}\label{lem:num_stability}
The numerical solutions $v_i$ are uniformly bounded in the $L^{\infty}$ norm. More precisely, there exists a finite constant $C>0$ such that
\begin{align} \label{eq:stab_bound}
\smnorm{v_i}_{L^{\infty}(S_i \times \oO)} \leq C \qquad \forall \, i \in \N.
\end{align}
\end{theorem}

\begin{proof}
Recall the solution $v_i^{\a}$ of the linear problem \eqref{eq:controlnumsol}. We define
\begin{align*}
\ik w & :=  P_i w^{\a}(\ik s, \cdot),\\
\ik{\tilde{v}} & := \ik w - v_i^{\a}(\ik s, \cdot).
\end{align*}
It is convenient to set
\begin{align}\label{eq:stabA}
\ik e := \siha I \ik{\tilde{v}} + \siha E \iko{\tilde{v}} = \siha I \ik w + \siha E \iko w - \siha F.
\end{align}
Because of Assumptions~\ref{ass:ellproj} and \ref{ass:stab_existence}, the $w_i$ are uniformly bounded in the $L^\infty$ norm. Moreover, $0 \le v_i \le v_i^{\a}$ due to Theorem~\ref{thm:discretewellposedness}. Thus the statement of the theorem is proved once we demonstrate that the $v_i^{\a}$ are bounded from above independently of~$i$. This is equivalent to showing a lower bound for the $\ik{\tilde{v}}$. 

It follows from Lemma~\ref{lem:cons_err_convergence} that $e^k \ge 0$ for $i$ larger than some constant $M$. It is trivial that there exists a constant $C$ such that the inequality of \eqref{eq:stab_bound} holds for all $i \le M$. We therefore may assume w.l.o.g.~that $e^k \ge 0$ throughout. By possibly modifying $w$ through the addition of a positive constant we can assume that
\[
w_i^{T/h_i}  = P_i w(T, \cdot) \ge \| v_T(T, \cdot) \|_{L^\infty(\O)}
\]
while maintaining the supersolution property of $w$ because $c_{\O}^{\a}, c_{\pO}^{\a} \ge 0$. This implies $\tilde{v}_i^{T/h_i} \ge 0$, i.e.~the non-negativity at the final time.

Now suppose $\iko{\tilde{v}} \ge 0$. Then
\[
\ik{\tilde{v}} = (\siha I)^{-1} \bigl( \ik e - \siha E \iko{\tilde{v}} \bigr) \ge 0
\]
because also $(\siha I)^{-1} \ge 0$ and $- \siha E \iko{\tilde{v}} \ge 0$. Now induction in $k$ completes the proof.
\end{proof}

\section{Uniform Convergence} \label{sec:conv}
Analogously to the envelopes of functions introduced in Section \ref{sec:bvp} we define envelopes of the numerical solutions as follows
\[
\overline{v}(t,x) = \limsup_{i \to \infty} \sup_{(\ik s,\il y) \to (t,x)} v_i(\ik s,\il y), \qquad  \underline{v}(t,x) = \liminf_{i \to \infty} \inf_{(\ik s,\il y) \to (t,x)} v_i(\ik s,\il y)
\]
where limits are taken over all sequences of nodes in $[0,T] \times \oO$ which converge to $(t,x) \in [0,T] \times \oO$.
Owing to Theorem \ref{lem:num_stability}, $\overline{v}$ and $\underline{v}$ attain finite values. By construction, $\overline{v}$ is upper and $\underline{v}$ lower semi-continuous and $\underline{v} \le \overline{v}$.

\begin{theorem}
The function $\overline{v}$ is a viscosity subsolution and $\underline{v}$ is a viscosity supersolution.
\end{theorem}

\begin{proof}
{\em Step 1 ($\overline{v}$ is a subsolution).}  To show that $\overline{v}$ is a viscosity subsolution, suppose that $w \in C^{\infty}(\R \times \R^d)$ is a test function such that $\overline{v}-w$ has a strict local maximum at $(s,y) \in (0,T)\times\oO$, with $ \overline{v}(s,y)=w(s,y)$. Note that $(s,y)$ may be on the boundary. Consider a closed neighbourhood $B := \bigl\{ (t,x) \in (0,T)\times \oO\; : \; |t-s|+|x-y| \leq \delta \bigr\}$ with $\delta>0$ such that
\[
\overline{v}(s,y)-w(s,y) > \overline{v}(t,x)-w(t,x) \quad \forall (t,x) \in B \setminus (s,y).
\]
Choose $i$ sufficiently large for $B$ to contain nodes. As in \cite{max_SIAM} we choose a sequence of nodes $\{ (s_{i(j)}^k,y_{i(j)}^{\ell}) \}_j$ which maximise $v_i(s_{i(j)}^\kappa,y_{i(j)}^\lambda) - P_i w(s_{i(j)}^\kappa,y_{i(j)}^\lambda)$ among all nodes $(s_{i(j)}^\kappa,y_{i(j)}^\lambda) \in B$ and converge to $(s,y)$.
It follows that
\begin{align} \label{eq:defmu}
v_i(\ik s,\il y)-P_i w(\ik s,\il y) \tends \overline{v}(s,y)-w(s,y)=0.
\end{align}
Moreover, because of $(\ik s,\il y) \tends (s,y)$, the neighbours of the $(\ik s,\il y)$ eventually also belong to $B$: for $i$ sufficiently large, we have $
(\ika s,\ila y) \in B $ if $\kappa \in \left\{k,k+1\right\}$ and $\ila y \in \supp \hil \phi$, in which case
\begin{align*}
v_i(\ika s,\ila y) - P_i w (\ika s,\ila y) \leq v_i(\ik s,\il y) - P_i w (\ik s,\il y) \\ \Leftrightarrow \; P_i w (\ika s,\ila y) + \mu_i \geq v_i  (\ika s,\ila y),
\end{align*}
with $\mu_i =  v_i(\ik s,\il y) - P_i w (\ik s,\il y)$, and $\mu_i \to 0$ as $i \to \infty$ because of~\eqref{eq:defmu}.

Recall that the matrices $\sia E$ have non-zero off diagonal entries $ \left(\sia E\right)_{\ell \lambda}$ only if $ \ila y \in \supp \hil \phi$ and that $v_i(s_i^{k+1},\cdot) \leq P_i w(s_i^{k+1},\cdot)+\mu_i$ on $\supp \hil \phi$. Therefore, monotonicity of $h_i \sia E - \Id$ for all $\a \in A$ implies that
\[
\left( (h_i \sia E - \Id ) \left[ P_i w (s^{k+1}_i,\cdot)+\mu_i \right] \right)_{\ell} \leq \left( (h_i \sia E-\Id) v_i(s_i^{k+1},\cdot)\right)_{\ell}.
\]
Applying the LMP and linearity of $ \sia I $ to $ P_i w (\ik s,\cdot) +\mu_i - v_i(\ik s,\cdot)$, which has a non-positive local minimum at $ \il y$, yields
\[
\left( (h_i \sia I + \Id ) \left[P_i w(\ik s,\cdot)+\mu_i \right]\right)_{\ell} \leq \left( (h_i \sia I + \Id) v_i(\ik s,\cdot)\right)_{\ell}.
\]
From the definition of the scheme, with $\gamma := \sup_{\alpha,i} \| \bia c + \bbia c \|_\infty$,
\begin{align} \nonumber
0 = \, & - d_i v_i (\ik s,\il y) + \sup_{\a \in A} \left( \sia E v_i(s_i^{k+1},\cdot)+\sia I v_i(\ik s,\cdot) - \sia F \right)_{\ell}  \\ \nonumber
\geq \, & - d_i \! \left( P_i w (\ik s,\il y) + \mu_i \right) + \sup_{\a\in A} \! \left( \sia E \! \left( P_i w (s_i^{k+1} ,\cdot) +  \mu_i \right) + \sia I \! \left( P_i w (\ik s,\cdot) + \mu_i \right) - \sia F \right)_{\ell}\\
\nonumber
= \, & - d_i P_i w(\ik s, \il y) + \sup_{\a \in A} \left[ \left( \sia E P_i w (s_i^{k+1} ,\cdot) + \sia IP_i w (\ik s,\cdot) - \sia F \right)_{\ell} + \mu_i \pair{\bia c+\bbia c}{\hil \phi}\right] \\
\geq \, & - d_i P_i w(\ik s, \il y) + \sup_{\a \in A} \left( \sia E P_i w (s_i^{k+1} ,\cdot) + \sia IP_i w (\ik s,\cdot) - \sia F \right)_{\ell} - \gamma \abs{\mu_i}\nonumber \\
= \, & F_i P_i w(\ik s, \il y) - \gamma \abs{\mu_i}. \label{eq:subsolineq1}
\end{align}
For a fixed $i$, evaluating $F_i P_i w(\ik s, \il y)$ may involve a boundary operator even if $(s,y)$ is internal and vice versa may involve the PDE operator even if $(s,y)$ belongs to the boundary. Referring to the semi-continuous envelope $F_*$, it now follows from \eqref{eq:subsolineq1}, $\lim_i \mu_i = 0$ and Theorem \ref{thm:consistency} that
\begin{align*}
0 & \geq \liminf_{i \to \infty} F_i P_i w(\ik s, \il y)\\
& \ge F_*(t, y, \Delta w(s,y), \nabla w(s,y), \partial_t w(s,y), \overline{v}(s,y)).
\end{align*}
Therefore $\overline{v}$ is a viscosity subsolution.

{\em Step 2 ($\underline{v}$ is a supersolution).} Arguments similar to those above show that $\underline{v}$ is a viscosity supersolution, where the principal change to the proof is that one considers $w \in C^{\infty}(\R \times \R^d)$ such that $\underline{v}-w$ has a strict local minimum at some $(s,y)\in(0,T)\times\O$ with $\underline{v}(s,y)=w(s,y)$. With analogous notation, the last line in \eqref{eq:subsolineq1} corresponds to
\[
0 \leq -d_i P_i w(\ik s,\il y)+\sup_{\a\in A}\left(\sia EP_i w(s_i^{k+1},\cdot)+\sia I P_i w(\ik s,\cdot)-\sia F\right)_{\ell}+ \gamma \abs{\mu_i},
\]
i.e. there is a slight asymmetry in the argument due to the last sign in \eqref{eq:subsolineq1}. Nevertheless, it is then deduced that
\[
0 \leq F^*(t, x, \Delta w(t,x), \nabla w(t,x), \partial_t w(t,x), \underline{v}(t,x)).
\]
Thus $\underline{v}$ is a viscosity supersolution.
\end{proof}

The above proof is an adaptation of the Barles-Souganidis argument \cite{barles_souganidis} to the finite element setting, in line with that in \cite{max_SIAM} but differing in the treatment of the boundary conditions.

\begin{assumption} \label{ass:comp}
Let $\overline{v}$ be a lower semi-continuous supersolution and $\underline{v}$ be an upper semi-continuous subsolution. Then $\underline{v} \le \overline{v}$.
\end{assumption}

\begin{theorem} \label{thm:uniform}
One has $\underline{v} = \overline{v} = v$, where $v$ is the unique viscosity solution with $v(T,\cdot)=v_T$.  Furthermore
\begin{align} \label{conv}
\lim_{i \to \infty} \| v_i - v \|_{L^\infty((0,T) \times \O)} = 0.
\end{align}
\end{theorem}

\begin{proof}
Follows as in the proof of Theorem 6.2 in \cite{max_SIAM}.
\end{proof}

\section{Numerical experiments} \label{sec:numexp}

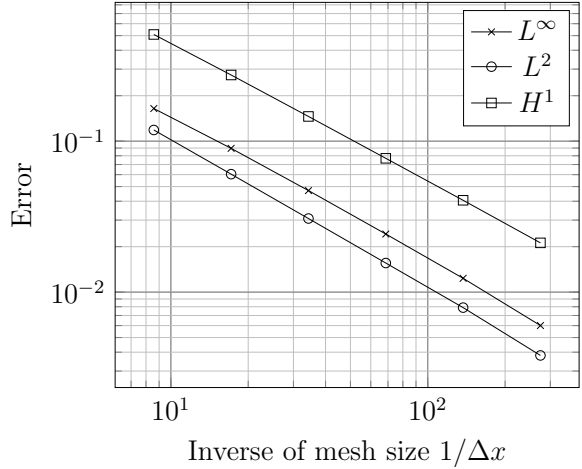
\begin{figure}
\begin{center}
\begin{tikzpicture}[scale=0.9]
\begin{loglogaxis}[xlabel = $\textrm{Inverse of mesh size} \ 1/\Delta x$,
ylabel=Error,grid=both,major grid style={black!50}]
\addplot [color=black,mark=x] coordinates {  
	(8.58156939882731,0.16449567037056068)   
	(17.1631387976546,0.08959787332893066)  
	(34.32627759530913,0.047057328161504874) 
	(68.65255519061786,0.024256140867236042) 
 	(137.30511038123396,0.012344159086025153) 
	(274.61022076246036,0.0060056651263776795)	
};
\addlegendentry{$L^{\infty}$}

\addplot [color=black,mark=o] coordinates {
	(8.58156939882731,0.11855766531960708)
	(17.1631387976546,0.06044214589853559)
	(34.32627759530913,0.030721496351280964)     
	(68.65255519061786,0.015582164914416639)
 	(137.30511038123396,0.007894432881552253)
	(274.61022076246036,0.003806337122879457)
};
\addlegendentry{$L^2$}

\addplot [color=black,mark=square] coordinates {
	(8.58156939882731,0.5088802898111241)   
	(17.1631387976546,0.2742997337502777)   
	(34.32627759530913,0.14573198702745077)   
	(68.65255519061786,0.07696854826437785)  
 	(137.30511038123396,0.04057651992064218)
	(274.61022076246036,0.021196988242819452)
};
\addlegendentry{$H^1$}
\end{loglogaxis}
\end{tikzpicture}
\caption{Approximation error of Experiment 1}
\label{fig:error_plot}
\end{center}
\end{figure}

The first experiment investigates rates of convergence for a known smooth solution. The remaining experiments examine the approximation of solutions with singularities near type changes of boundary conditions as well as the solution behaviour in the vicinity of nonlinear boundary conditions. The code is available from the public repository \cite{github} under the GNU Lesser General Public License.

{\bf Experiment 1 (Rates for smooth known solution):} 
We consider a final time boundary value problem on the square domain $\O = [-1,1]^2$ with Robin conditions on the right face $\pO_t = \{1\} \times (-1,1)$ and Dirichlet conditions on the remaining faces $\pO_D = \pO \setminus \pO_t$. We have the control set $A=[0,1]$ and the final time $T=1$ for the system
\begin{equation}
\label{eq:experiment1}
\begin{aligned}
-\p_t v + \sup_{\a \in A}\left(-(\a+|x|^2/2) \Delta v + xv_x - f^{\a}\right)    & = 0 	&	&\quad\text{in }[0,T)\times\O,\\
-\p_t v + \sup_{\a\in A} \bigl( \a v_x - g^{\a} \bigr)    & = 0 	&	& \quad \text{on }[0,T)\times\pO_t,\\
v \;\;\;  & = 0	&	&\quad\text{on }[0,T)\times\pO_{D},\\[2mm]
v - (1-x^2)(1-y^2) \,    & = 0	&	&\quad \text{on }\{T\}\times\oO.
\end{aligned}
\end{equation}
We choose $g^{\a}$ and $f^{\a}$ such that
\[
v(x,y,t) := t(1-x^2)(1-y^2) + (1-t)\sin{(\pi x)}\cos{\left(\frac{\pi y}{2}\right)}
\]
is the exact solution of \eqref{eq:experiment1}.

The artificial diffusion coefficients are selected quasi-optimally, cf.~Section \ref{sec:scaling}. The time dependent Robin boundary condition is treated fully explicitly.  The time step size is chosen to ensure monotonicity while permitting a large time step, leading to $O(h_i)=O(\Delta x_i)$. The $L^2$, $H^1$ and $L^{\infty}$ errors at time $t=0$, presented also in Figure \ref{fig:error_plot}, obey in essence the same rates as those observed previously \cite{max_SIAM} with Dirichlet conditions and $O(h_i)=O(\Delta x_i)$ scaling:
\begin{center}
\small \begin{tabular}{c | c c | c c | c c} 
 $\Delta x$ & $L^2$ & Rate & $L^{\infty}$ & Rate & $H^1$ & Rate \\ [0.5ex] 
 \hline
 0.1165 & 1.186\text{e-}1 & 0.98 & 1.645\text{e-}1 & 0.92 & 5.089\text{e-}1 & 0.93 \\ 
 0.0583 & 6.044\text{e-}2 & 0.98 & 8.960\text{e-}2 & 0.95 & 2.743\text{e-}1 & 0.94 \\ 
 0.0291 & 3.072\text{e-}2 & 0.99 & 4.706\text{e-}2 & 0.97 & 1.457\text{e-}1 & 0.95 \\ 
 0.0146 & 1.558\text{e-}2 & 0.99 & 2.426\text{e-}2 & 0.98 & 7.696\text{e-}2 & 0.95 \\ 
 0.0073 & 7.894\text{e-}3 & 1.04 & 1.234\text{e-}2 & 1.03 & 4.058\text{e-}2 & 0.96 \\ 
 0.0036 & 3.806\text{e-}3 &      & 6.006\text{e-}3 &      & 2.120\text{e-}2 &      \\  [1ex] 
\end{tabular}
\end{center}

\begin{figure}[ht]
\centering
\includegraphics[width=0.65\textwidth]{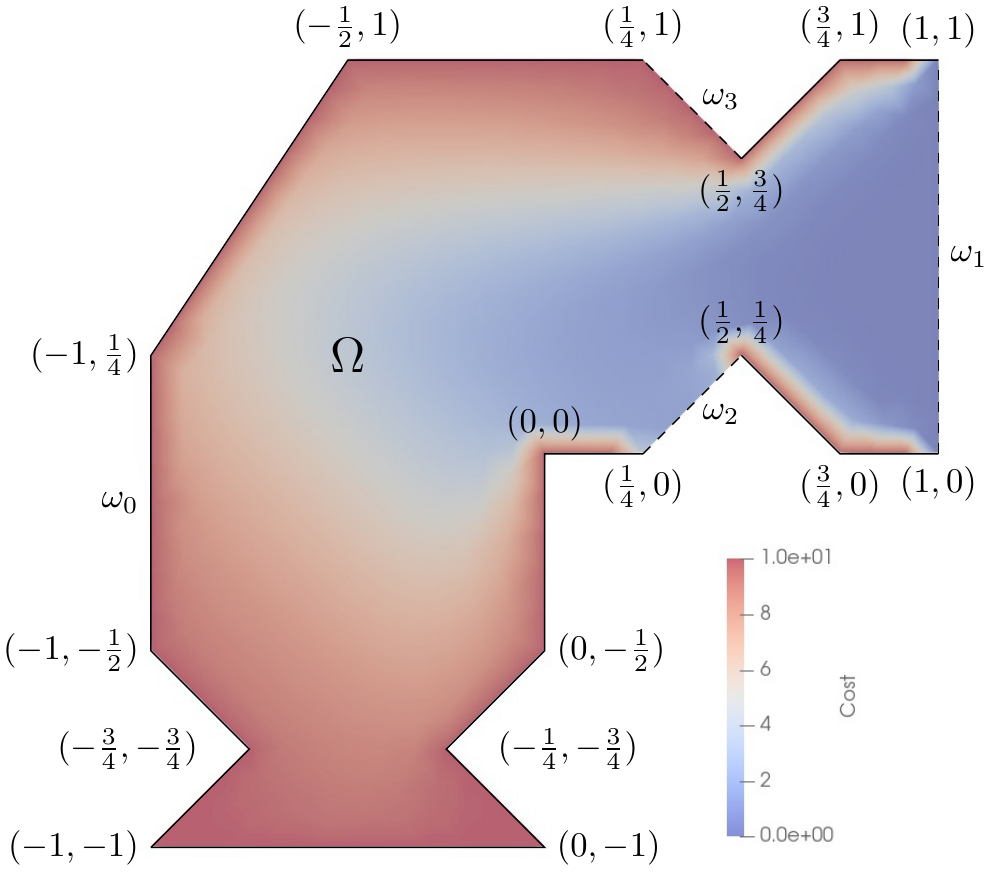}
\caption{Value function of the Skorokhod problem for a coarse mesh size $\Delta x \approx 0.12$}
\label{fig:exp2}
\end{figure}

{\bf Experiment 2 (Skorokhod problem):}
The second numerical experiment is set on a non-convex, less regular domain, which is depicted in Figure~\ref{fig:exp2}. The stochastic controlled process is subject to a terminal cost of $10$ everywhere apart from $\overline{\omega_1}$ where it is 0. There is no running cost. On $\omega_2$ the particle is transported through a Skorokhod reflection independently of the angle of incidence in the direction of the inner normal vector. Ultimately, a particle can avoid penalisation only by reaching $\overline{\omega_1}$ before the terminal time $T = 1$. On $\Omega$ the particle may only choose between an upwards drift and drift to the right:
\begin{subequations}
\begin{align} 
-\p_t v + \sup_{\a}\left(-a^{\a} \Delta v - b^{\a} \cdot \nabla v \right) & = 0 && \textstyle \text{in } [0,T)\times\O,\\ \label{eq:omega2BC}
- b_{\pO} \cdot \nabla v & = 0 & &\textstyle \text{on }[0,T) \times \omega_2 \cup \left\{\left(\frac{1}{4}, 0\right)\right\},\\[1mm] \label{eq:experiment2}
v & = 0	& & \textstyle \text{on }[0,T)\times \overline{\omega_1},\\[1mm] 
v & = 10 & & \textstyle \text{on }[0,T)\times \omega_0 \cup \overline{\omega_3} \cup \left\{\left(\frac{1}{2}, \frac{1}{4}\right)\right\},\\[1mm] 
v & = v_T &	& \textstyle \text{on }\{T\}\times\oO,
\end{align}
\end{subequations}
where $a^{\a} = 0.1(1-x_2)\a$ and $b^{\a} = (- 2\a, 2(\a - 1))^T$ for $\a \in \{0,1\}$. Moreover, $b_{\pO} = (1, -1)^T$ and
\begin{align} \label{eq:vT}
v_T(x) =  \begin{cases} 10 \quad &x \in \oO \setminus \omega_1, \\ 0 \quad & x \in \omega_1. \end{cases}
\end{align}
Hence when drifting to the right the particle is exposed to Brownian noise, while the equation is degenerate when the upward drift is selected. The numerical operators are given by
\begin{subequations}
\begin{align} \label{eq:exp2a}
(\siao E v)_\ell &:= \, \bialo  \nu \langle \nabla v, \nabla \hil \phi \rangle + \langle -b^{\a} \cdot \nabla v, \hil \phi \rangle \\ \label{eq:exp2b}
(\siao I v)_\ell &:= \, \max\left(a^{\a} - \bialo \nu, 0\right) \langle \nabla v, \nabla \hil \phi \rangle. \\
(\siar E v)_\ell &:= 0, \\
(\siar I v)_\ell &:= \frac{v(t,x_1, x_2) - v(t,x_1 - \lambda, x_2 + \lambda)}{\lambda}.
\end{align}
\end{subequations}
Figure \ref{fig:exp2} shows the approximation $v_i$ at time $t=0$ on a coarse mesh. Notice how the introduction of a Skorokhod type boundary gives the particle starting in the vicinity of $\omega_2$ a high probability of reaching the penalty-free exit zone $\overline{\omega_1}$. The node at $(\frac{1}{2},\frac{1}{4})$ already belongs to the Dirichlet boundary. We observe that the penalty of $10$ in the boundary segment between $\omega_1$ and $\omega_2$ leads to a layer-like behaviour of the solution of only one element thickness. The related numerical experiments on finer meshes depicted in Figure \ref{fig:exp3} also exhibit this aspect of the numerical solution.

\begin{figure}[t]
     \centering
     \begin{subfigure}{0.48\textwidth}
         \centering
         \includegraphics[width=\textwidth]{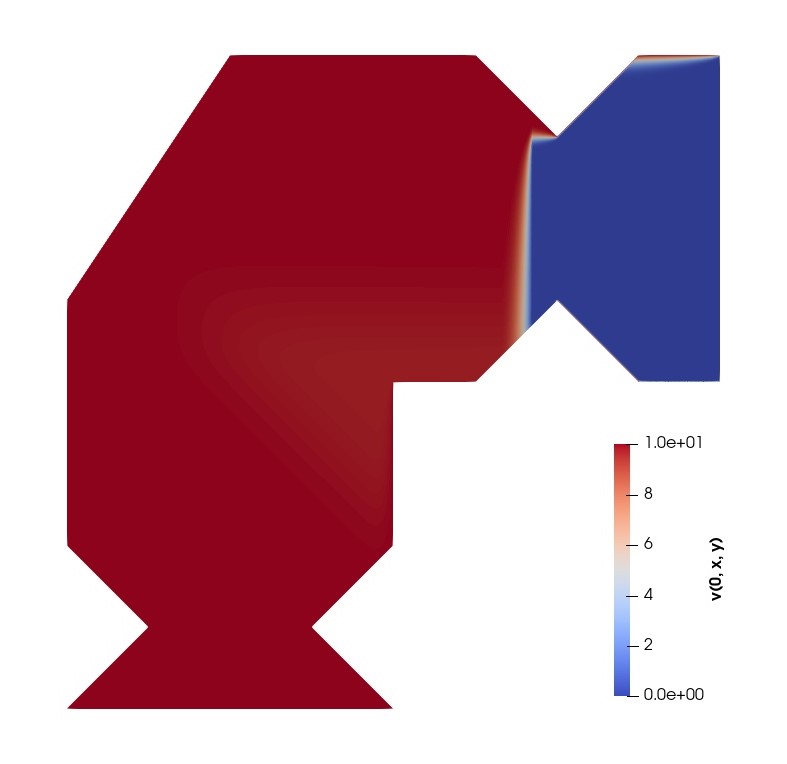}
         \caption{Skorokhod problem}
         \label{fig:exp3a}
     \end{subfigure}
     \hfill
     \begin{subfigure}{0.48\textwidth}
         \centering
         \includegraphics[width=\textwidth]{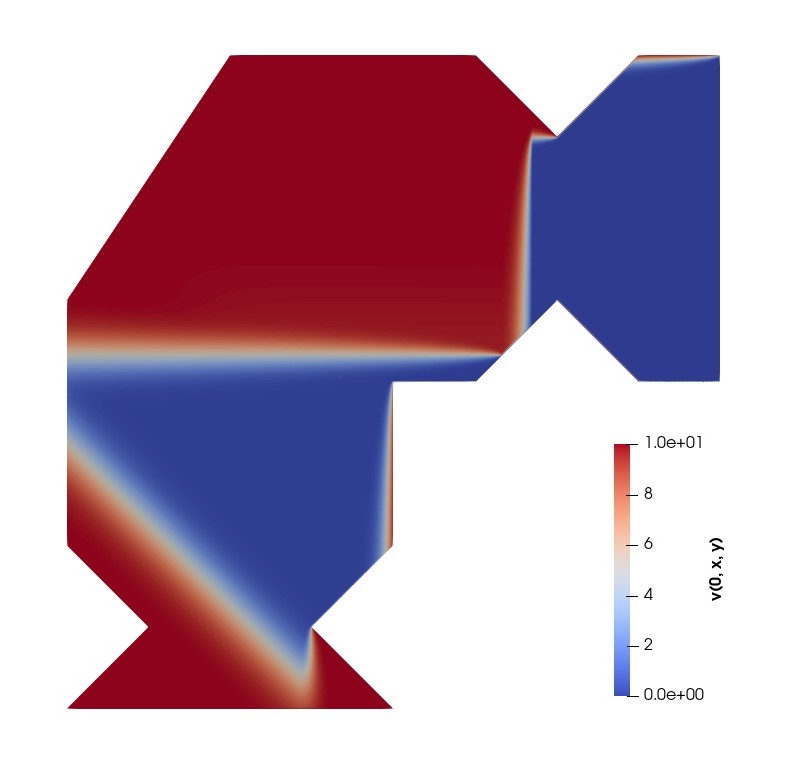}
         \caption{Nonlinear boundary condition}
         \label{fig:exp3b}
     \end{subfigure}
     \hfill
    \caption{Linear and nonlinear boundary conditions on $\omega_2$ with $\Delta x = 0.0035$.}
    \label{fig:exp3}
\end{figure}

{\bf Experiment 3 (Internal barrier and nonlinear boundary conditions):} 
The third experiment is an adaptation of the previous one to examine the effect of nonlinear boundary conditions. We break the adaptation into two parts.

{\em Part (a) (Internal barrier)}: The experiment is identical to the previous one with exception of the drift terms on $\Omega$:
\begin{equation*}
b^{\a}(x) = 
\begin{cases}
\bigl(- 2\a, 2(\a-1)\bigr)^T \quad & : |x-\frac{3}{8}|>\frac{1}{20},\\
\phantom{-\a}\bigl(0, 2(\a-1)\bigr)^T  \quad & : \textrm{otherwise.}
\end{cases}
\end{equation*}
This means that the strip of all $x$ with $|x-\frac{3}{8}| \le \frac{1}{20}$ acts as barrier in $\Omega$: Within this strip there is no process with drift to the right. The only way a particle can cross the strip from the left to the right in order to avoid penalisation is by adopting $\alpha = 1$. In this case the particle {\em might} cross the barrier by means of diffusion; however, there is no drift term to aid the crossing.

The construction results in a value function which at first sight resembles a piecewise constant function. It is close to 10 left of the barrier as the particle is unlikely to reach the penalty-free exit zone $\omega_1$. It is mostly close to 0 right of the barrier; however, reaches 10 at Dirichlet boundary conditions on $\omega_0$ as already indicated in Experiment 2. At the barrier there is an internal layer arising from the possibility of crossing owing to diffusion. 

This description of the value function is matched entirely by the numerical computations with the scheme of this paper: see Figure \ref{fig:exp3a}, where the internal layer as well as the boundary conditions right of the barrier are well resolved.

{\em Part (b) (Nonlinear boundary condition)}: Now the boundary condition on $\omega_2$ is replaced by a nonlinear operator which corresponds to the choice between the previously used Skorokhod reflection and instantaneous transport along the boundary towards the right. In other words, \eqref{eq:omega2BC} is replaced by
\[
\sup \{ - b^0_{\pO} \cdot \nabla v,  - b^1_{\pO} \cdot \nabla v \} = 0 \quad \textstyle \text{on }[0,T) \times \omega_2 \cup \left\{\left(\frac{1}{4}, 0\right)\right\}.
\]
with $b^0_{\pO} = (1,-1)$ and $b^1_{\pO} = (-1,-1)$.

This modification has a striking impact on the behaviour of the controlled system. Now an optimally controlled particle may drift onto the boundary segment $\omega_2$ left of the barrier to be then transported on the boundary past the barrier. Right of the barrier the control is changed to $b^0_{\pO}$ in order to reflect the particle into $\Omega$ to avoid the penalty at the point $(\frac{1}{2}, \frac{1}{4})$ at the end of $\omega_2$.

An approximation of the resulting value function at time $0$ is shown in Figure \ref{fig:exp3b}, which illustrates how this time regions left of the barrier have a value function close to zero since particles located there can now reach the penalty-free exit zone $\omega_1$. Indeed, when computing the solutions for earlier times $t < 0$, one observes further growth of the blue region as more time is available to arrive at $\omega_1$ before termination.

{\bf Experiment 4 (Reflection vs.~termination):} 
We consider a final time boundary value problem on the same domain, but now with a nonlinear boundary condition corresponding to a choice between a Skorokhod reflection and termination of the process in exchange for an oscillatory cost $g^\alpha$:
\begin{align*}
-\p_t v + \sup_{\a} \left(-a^{\a} \Delta v - b^{\a} \cdot \nabla v \right) & = 0 & &\quad\text{in }[0,T)\times\O,\\
\sup_{\a} \bigl( - b_{\pO}^{\a} \cdot \nabla v + c^{\a}_{\pO} v - g^{\a} \bigr) & = 0	& & \quad \textstyle \text{on }[0,T) \times \omega_3 \cup \left\{\left(\frac{1}{4}, 1\right)\right\},\\
v & = 0	& & \quad\text{on }[0,T)\times \overline{\omega_1},\\
v & = 10 & & \textstyle \quad\text{on }[0,T)\times \omega_0 \cup \overline{\omega_2} \cup \left\{\left(\frac{1}{2}, \frac{3}{4}\right)\right\},\\
v & = v_T &	&\quad \text{on }\{T\}\times\oO,
\end{align*}
where $\a \in \{0,1\}$, $v_T$ as in \eqref{eq:vT} and
\begin{align*}
    a^{\a} & = 0.2(1-x_2)(1-\a) + 0.2(1-x_1)\a,\\
    b^{\a} & = (- 2 \a, 2 (\a - 1) )^T, \quad b^{\a}_{\pO} = (\a, \a )^T,\\
    c^{\a} & = (1-\a), \quad g^{\a} = -(10\cos(160 x_1 / \pi+4)(1-\a).
\end{align*}
Note that compared to the two previous examples the Robin type boundary has moved to $\omega_3$ while $\omega_2$ is part of the Dirichlet region with value $10$. Both operators $L^\a$ now have regions of degeneracy, when either $x_1$ or $x_2$ is near $1$. 

The PDE operator on $\Omega$ is discretised according to \eqref{eq:exp2a}--\eqref{eq:exp2b}, while the Robin operators are approximated implicitly, consistent with Assumption \ref{ass:consistency}. The behaviour of the value function $v$ in the vicinity of $\omega_3$ is depicted in Figure~\ref{fig:exp4}. One observes how troughs of $g^0$ are attained by the value function, while near peaks of $g^0$ the numerical scheme switches to the reflection principle. Overall the experiment demonstrates how the framework of the paper not only allows us to approximate nonlinear Robin conditions, but also incorporates a nonlinear switching between Robin conditions on the one hand and Dirichlet conditions on the other hand.

\begin{figure}[t]
\centering
\includegraphics[width=\textwidth]{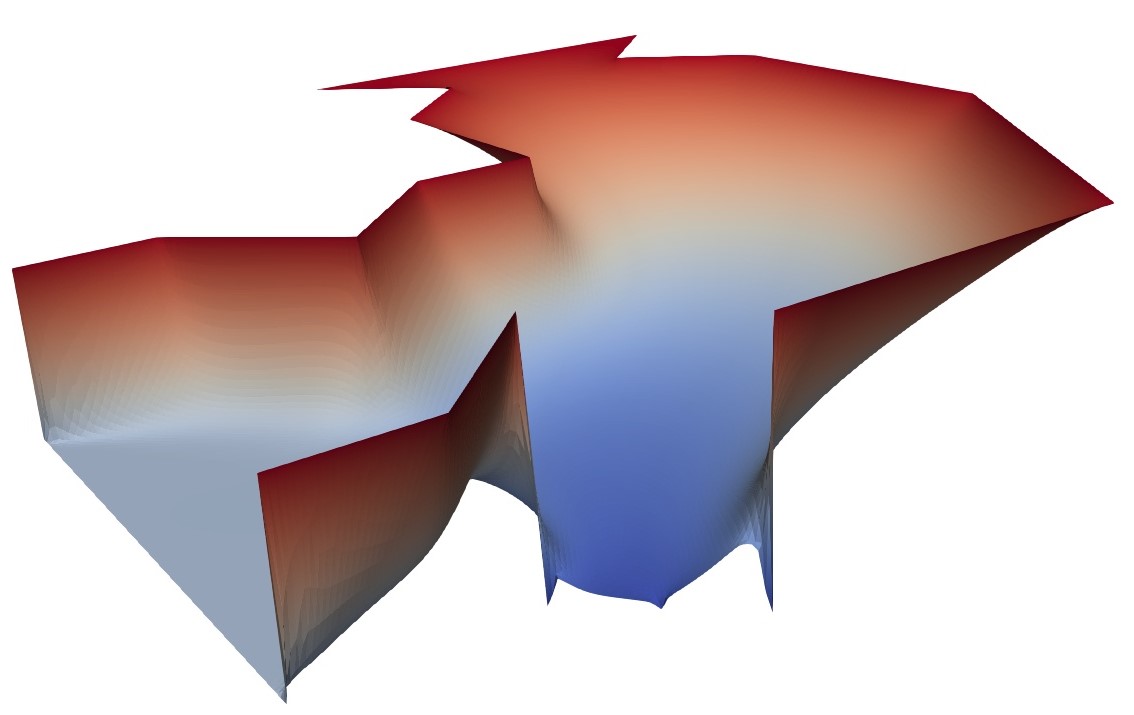}
\caption{Value function with nonlinear boundary condition on $\omega_3$.}
\label{fig:exp4}
\end{figure}


\section*{Acknowledgements}

Bartosz Jaroszkowski gratefully acknowledges the support of the EPSRC grant 1816514. Max Jensen gratefully acknowledges the support of the Dr Perry James Browne Research Centre.


\bibliographystyle{alpha}
\bibliography{references}


\appendix
\section{Appendix: Interpretation of mixed boundary conditions} \label{sec:interpretation}

\small
We briefly sketch in a simplified setting how mixed boundary conditions can arise from an underlying optimal control problem. For a start we assume here that the solution $v$ of \eqref{eq:Bellmanibvp} is smooth. In optimal control formulations the coefficients $c^\alpha$ and $c_{\pO}^\alpha$ are typically either $0$ or they all coincide with some constant in order to model a discounting of cost. We shall assume the former. 

We consider a particle, or agent, which occupies the state $\bx (t) \in \oO$ at time $t \in [0,T]$. Its movements are described by the following rules:
\begin{enumerate}
\item Suppose $\bx (t) \in \Omega$, $t < T$ and the control $\alpha \in A$ is selected. Then the particle's immanent movement is described by the SDE
\begin{align} \label{eq:SDE}
\d \bx = b^\alpha(\bx) \, \d t + \sqrt{2 \, a^\alpha} \, \d W,
\end{align}
where $W$ is a $d$-dimensional Brownian motion. While the particle follows \eqref{eq:SDE} it is subject to the cost $f^\alpha \d t$.
\item Suppose $\bx (t) \in \pO_D^\circ$, $t < T$ and the control $\alpha \in A$ is selected. Here $\pO_D^\circ$ refers to the interior of $\pO_D$ relative to $\pO$. Then, in the context of viscosity boundary conditions, the particle may either follow \eqref{eq:SDE} at the running cost $f^\alpha \d t$ or terminate its movement at a cost of $g(\bx(t))$. If instead pointwise Dirichlet conditions were imposed in Definition~\ref{def:vissol} then the boundary conditions would correspond to the guarantee that the particle terminates its movement.
\item  Suppose $\bx (t) \in \pO_R^\circ$, $t < T$ and the control $\alpha \in A$ is selected. Then the Skorokhod reflection principle may apply. Indeed, as described in \cite[Sections 1.4, 3.1.3, \ldots]{FDbook}, upon reaching the boundary the particle may instantaneously be transported a distance $b_{\pO}^\alpha \, \delta$ where $\delta > 0$ is small. Alternatively, because of the context of viscosity boundary conditions, the particle may continue to follow \eqref{eq:SDE}. To remain within the scope of \cite{FDbook} we assume $g^\alpha = 0$ on $\pO_R$.
\item Suppose $\bx (t) \in \pO_t^\circ$, $t < T$ and the control $\alpha \in A$ is selected. Then the particle may either move according to
\begin{align} \label{eq:SDE_pO}
\d \bx = b_{\pO}^\alpha(\bx) \d t
\end{align}
with running cost $g^\alpha \d t$ or according to \eqref{eq:SDE} with running cost $f^\alpha \d t$.
\item Suppose $t = T$ and the particle's movement has not yet terminated at a Dirichlet boundary then the final time cost $v_T(\bx(T))$ incurs.
\item Suppose that $\bx (t) \in \partial (\pO_D) \cup \partial (\pO_R) \cup \partial (\pO_t)$, $t < T$. Here the outer $\partial$ of $\partial (\pO_X)$ refers to the boundary of $\pO_X$ relative to $\pO$ where $X \in \{ D, R, t\}$. Then the particle's behaviour may be selected from multiple of the above scenarios. E.g. if $\bx (t) \in \partial (\pO_D) \cup \partial (\pO_R)$ then the particle movement may terminate, the particle may be reflected or it may be transported according to \eqref{eq:SDE}.
\end{enumerate}
All these possible scenarios occur when representing uncertain market price of volatility risk in a Heston model found in \cite{HestonPaper}.

We now link the above description of the particle by means of the value function to the HJB final time boundary value problem \eqref{eq:Bellmanibvp}. Let $\ba: [0,T] \to A$ represent a choice of controls for each time $s \in [0,T]$. Similarly, let $\xi_\Omega, \xi_{\pO_D}, \xi_{\pO_R}, \xi_{\pO_t}: [0,T] \times \oO \to \{ 0, 1\}$ be indicator functions such that $\supp \xi_{\pO_X} \subset \overline{\pO_X}$ for $X \in \{D,R,t\}$, where $\overline{\pO_X}$ is the closure of $\pO_X$ relative to $\pO$. Furthermore,
\[
\xi_\Omega + \xi_{\pO_D} + \xi_{\pO_R} + \xi_{\pO_t} \equiv 1.
\]
Where $\xi_\Omega = 1$ the particle path $\bx$ obeys \eqref{eq:SDE}, where $\xi_{\pO_D} = 1$ the particle terminates, where $\xi_{\pO_R} = 1$ the particle is reflected and where $\xi_{\pO_t} = 1$ the particle follows~\eqref{eq:SDE_pO}. Since the particle terminates where $\xi_{\pO_D} = 1$ we requite $\xi_{\pO_D}(s_1,x) = 1 \Rightarrow \xi_{\pO_D}(s_2,x) = 1$ for $s_1 \le s_2$. The value function $v$ at $(t,x)$ is the smallest cost realised among all possible choices for $\ba$ and ${\boldsymbol \xi} = (\xi_\Omega, \xi_{\pO_D}, \xi_{\pO_R}, \xi_{\pO_t})$:
\begin{align*}
v(t,x) = 
\inf_{\ba, {\boldsymbol \xi}} {\bf E}_{xt} & \Bigl ( \int_t^\tau \xi_\Omega(\bx(s)) \, f^{\ba(s)}(\bx(s)) + \xi_{\pO_t}(\bx(s)) \, g^{\ba(s)}(\bx(s)) \d s\\
& \, + \xi_{\pO_D}(\bx(\tau)) \, g(\bx(\tau)) + \xi_\Omega(\bx(\tau)) \,  v_T(\bx(\tau)) \Bigr),
\end{align*}
where $\tau$ is the exit time from $[0,T) \times \Omega$ of $\bx$ and ${\bf E}_{xt}$ the expectation conditional to $\bx(t) = x$. 

We fix some $\ba, {\boldsymbol \xi}$ which are not necessarily optimal. Suppose that $\bx(s) \in \supp \xi_{\pO_t}$ for a short duration $[t,t+\eps) \ni s$. Then
\begin{align*}
g^{\ba(t)}(\bx(t)) & = \lim_{h \to 0} \frac{1}{h} \int_t^{t+h} g^{\ba(s)}(\bx(s)) \d s
\ge - \lim_{h \to 0} \frac{v(t+h, \bx(t+h)) - v(t,x)}{h}\\
& = - \partial_t v(t,\bx) - \nabla v(t,\bx) \cdot \dot{\bx} = - \partial_t v(t,\bx) - \nabla v(t,\bx) \cdot b_{\pO}^\alpha(\bx).
\end{align*}
Here the first equality follows from continuity, the inequality from the dynamic programming principle, the second equality from the chain rule and the third from \eqref{eq:SDE_pO}. 

Now, suppose that $\bx(s) \subset \supp \xi_\Omega$ for a short duration $[t,t+\eps) \ni s$. Then by a similar argument, detailed in \cite{FDbook}, one finds
\begin{align*}
f^{\ba(t)}(\bx(t)) \ge - \partial_t v(t,\bx) - \nabla v(t,\bx) \cdot b^\alpha(\bx) - a^\alpha \nabla v(t,\bx).  
\end{align*}
On $\supp \xi_{\pO_D}$ we find $g(\bx(t)) \ge v(t,\bx(t))$ as the minimal cost cannot be more than the cost of termination. 

Suppose now that the particle is located at $x \in \supp \xi_{\pO_R}$. It cannot be more beneficial for the particle to be at $x + b_{\pO}^\alpha \, \lambda$ as it will immanently be transported there. Thus $v(t,x) \le v(t,x + b_{\pO}^\alpha \, \lambda)$ and therefore, with $\lambda \to 0$,
\[
- b_{\pO}^\alpha(x) \cdot \nabla v(t,x) \le 0.
\]
When and where-ever the choice of $\ba, {\boldsymbol \xi}$ is optimal, the respective above inequality turns into an equality. With the compactness of $A$ and the continuous dependence of the coefficients on $\alpha$, such optimal controls exist. Therefore, taking suprema over $A$, one obtains that the value function solves \eqref{eq:Bellmanibvp}, at least conceptually, with boundary conditions in the viscosity sense. We refer here to the viscosity sense because the use of semi-continuous envelopes in Definition 1 is interpreted as permitting \eqref{eq:SDE} as transport law on all of the closure $\oO$ and as offering at the interfaces between boundary regions multiple boundary operators for the choice of the optimal strategy, like indicated in scenario 6 of the above list. We note that the choice between \eqref{eq:SDE} and the various boundary operators will in general be subject to some delicate restrictions, arising from the sub- and superjets. At the boundary these jets are increased in size compared to their counterparts in the domain interior \cite[Remark 2.7]{Crandall:1992ta}.


\end{document}